\documentclass[a4paper,11pt]{amsart}

\usepackage{amsthm}
\usepackage{amssymb}
\usepackage{amsmath}
\usepackage{mathabx}
\usepackage{graphicx}
\usepackage{pstricks}
\usepackage[utf8]{inputenc}
\usepackage[shortlabels]{enumitem}
\setenumerate{itemsep=0mm,topsep=0mm}

\numberwithin{theoscounter}{section}

\numberwithin{equation}{section}

\theoremstyle{plain}
\newtheorem{theorem}{Theorem}[section]
\newtheorem*{theorem*}{Theorem}
\newtheorem{lemma}[theorem]{Lemma}
\newtheorem{corollary}[theorem]{Corollary}

\newtheorem{proposition}[theorem]{Proposition}

\theoremstyle{remark}

\newtheorem*{remarks*}{Remarks}

\newtheorem*{acknowledgements}{Acknowledgments}

\theoremstyle{definition}

\newtheorem{definition}[theorem]{Definition}

\makeatletter
\renewcommand{\pod}[1]{\mathchoice
  {\allowbreak \if@display \mkern 18mu\else \mkern 8mu\fi (#1)}
  {\allowbreak \if@display \mkern 18mu\else \mkern 8mu\fi (#1)}
  {\mkern4mu(#1)}
  {\mkern4mu(#1)}
}

\newcommand{\C}{\mathbb{C}}

\newcommand{\Q}{\mathbb{Q}}
\newcommand{\R}{\mathbb{R}}

\renewcommand{\Re}{\operatorname{Re}}

\newcommand{\res}{\text{res}}

\newcommand{\af}{\mathfrak{a}}

\newcommand{\pf}{\mathfrak{p}}

\newcommand{\Sf}{\mathfrak{S}}
\newcommand{\zf}{\mathfrak{z}}
\newcommand{\Ac}{\mathcal{A}}
\newcommand{\Bc}{\mathcal{B}}
\newcommand{\Cl}{\mathcal{C}l}

\newcommand{\Kc}{\mathcal{K}}

\newcommand{\Pc}{\mathcal{P}}

\usepackage[parfill]{parskip}
\title{Critical zeros of lacunary $L$-functions}

\author{J.B. Conrey}
\address{American Institute of Mathematics, 360 Portage Ave, Palo Alto, CA 94306, USA.}
\email{conrey@aimath.org}

\author{H. Iwaniec}
\address{The ETH Institute for Theoretical Studies, Clausiusstrasse 47, 8092 Zürich, Switzerland,}
\address{Department of Mathematics, Rutgers University, New Brunswick, NJ 08903, USA.}
\email{iwaniec@math.rutgers.edu}

\begin{document}

\maketitle
\thispagestyle{empty}

\makeatletter
\@starttoc{toc}
\makeatother
\newpage
\thispagestyle{empty}
\hspace{1cm}\vspace{0.5cm}
\begin{acknowledgements}
  This project started years ago during numerous visits of the second author to the American Institute of Mathematics and two visits of the first author to Rutgers University. We thank both institutions for these opportunities and their support. Then the work continued when the second author enjoyed a Senior Fellow position at the Institute for Theoretical Studies -- ETH Zürich in August 2014 -- July 2015. He is happy to acknowledge the superb working conditions and generous support received from ETH-ITS while the project was in progress and completed. We are also grateful for support from the NSF grants DMS 1406981 and DMS 1101574. We thank Corentin Perret for technical help.
\end{acknowledgements}
\newpage

\section{Introduction}
\label{section:introduction}

We consider an $L$-function given by the Euler product
\begin{equation}
  \label{eq:1.1}
  L(s)=\prod_p(1-\lambda(p)p^{-s}+\kappa(p)p^{-2s})^{-1}
\end{equation}
with $|\lambda(p)|\le 2$ and $|\kappa(p)|\le 1$, so the product converges absolutely in the half-plane $\Re s>1$. Hence $L(s)$ has the absolutely converging Dirichlet series expansion
\begin{equation}
  \label{eq:1.2}
  L(s)=\sum_n \lambda(n)n^{-s}\hspace{0.5cm}\text{ if }\Re s>1,
\end{equation}
with multiplicative coefficients $\lambda(n)$ which are bounded by the divisor function $\tau(n)$. Moreover, we assume that $L(s)$ admits analytic continuation to the whole complex $s$-plane and it is holomorphic, expect possibly for a simple pole at $s=1$. Furthermore, $L(s)$ satisfies a standard functional equation which we write in the following form
\begin{equation}
  \label{eq:1.3}
  L(s)=X(s)\overline L(1-s)
\end{equation}
where $\overline L(s)$ stands for the $L$-function with Dirichlet series coefficients complex conjugated and $X(s)$ is called the root factor. Note that $|X(s)|=1$ if $\Re s=1/2$. One may consider the equation \eqref{eq:1.3} as a definition of $X(s)$. Typically $X(s)$ turns out to be an exponential function times the ratio of one or two gamma functions. We do not need to specify the root factor. For our purpose it suffices to assume that $X(s)$ is holomorphic in the strip $0\le \Re{s} <1$ and it satisfies
\begin{equation}
  \label{eq:1.4}
  X(s+z)=X(s)(Q|s|)^{-2z}\left\{1+O\left(|z|(|s|+|z|)^{-1}\right)\right\}
\end{equation}
if $\Re s=1/2$ and $-1/4\leq \Re z\leq 0$, where the implied constant depends only on the parameters (shifts) in the involved gamma functions. In specific cases \eqref{eq:1.4} follows by Stirling's formula.

We say that $L(s)$ is ``lacunary'' if its coefficients vanish or are quite small frequently. We measure this phenomenon by postulating the following estimate
\begin{equation}
  \label{eq:1.5}
  \sum_{Q^4<n\le Q^{4A}} |\lambda(n)|n^{-1}\le \varepsilon A
\end{equation}
to hold with some $\varepsilon=\varepsilon(Q)>0$ for all $A\ge 1$. Here we think of $\varepsilon=\varepsilon(Q)$ being arbitrarily small as $Q$ gets large ($Q^2$ is closely related to the conductor of $L(s)$). However, even a fixed $\varepsilon>0$ but sufficiently small would suffice for nice applications. If \eqref{eq:1.5} holds, then we say that the $L(s)$ is ``$\varepsilon$-lacunary''.

\begin{remarks*}
  The lacunarity condition with $\varepsilon$ relatively small reveals that the coefficients $\lambda(n)$ of $L(s)$ appear less often than the primes numbers in the segments $[Q^4,Q^{4A}]$. When it comes to perform some mollification, this property means that one applies sieve of small dimension. But it is hard to believe that such $L$-functions do exist in reality, therefore our undertaking here is mainly for learning the phenomena and exercising delicate techniques. For instance we estimate the lacunary bilinear form \eqref{eq:22.4} without losing vital savings from sifting effects along the lines \eqref{eq:22.6}--\eqref{eq:22.10}. We are mostly interested in special $L$-functions, nevertheless we set the above introduction in some generality, because it exposes the ``exceptional characters'' at work more clearly than the roundabout argument with ``exceptional zeros''.
\end{remarks*}

Our primary source of lacunary $L$-functions is the quadratic field $K=\Q(\sqrt{D})$ of discriminant $D$. Let $\psi:\Cl(K)\to\C$ be a character of the ideal class group of $K$. There are $h(D)=|\Cl(K)|$ such characters. For each of these we have the $L$-function
\begin{equation*}
  L(s,\psi)=\sum_\af \psi(\af)(N\af)^{-s}=\prod_\pf (1-\psi(\pf)(N\pf)^{-s})^{-1}
\end{equation*}
which satisfies our conditions with
\begin{equation}
  \label{eq:1.6}
  \lambda(n)=\sum_{N\af=n}\psi(\af)
\end{equation}
and
\begin{equation}
  \label{eq:1.7}
  Q=\sqrt{|D|}/2\pi.
\end{equation}
In this case, the root number (the sign of the functional equation) is 1 and the root factor $X(s)$ is equal to $Q^{1-2s}\Gamma(1-s)/\Gamma(s)$ if $D<0$ and $Q^{1-2s}\Gamma^2(\frac{1-s}{2})/\Gamma^2(\frac{s}{2})$ if $D>0$.

The coefficients of $L(s,\psi)$ are bounded by the coefficients of $L(s,\psi_0)$ for the trivial character $\psi_0=1$. By the factorization $L(s,\psi_0)=\zeta(s)L(s,\chi)$, where $\chi\pmod{|D|}$ is the Dirichlet real character (given by the Kronecker symbol associated with the field $K=\Q(\sqrt{D})$) one sees that the coefficients of $L(s,\psi_0)$ are
\begin{equation}
  \label{eq:1.8}
  \lambda_0(n)=(1\star\chi)(n)=\sum_{d\mid n}\chi(d).
\end{equation}
Clearly, $|\lambda(n)|\le\lambda_0(n)\le\tau(n)$. Moreover, we have (cf. (22.109) of \cite{IK})
\begin{equation*}
  \sum_{n\le N} \lambda_0(n)n^{-1}=L(1,\chi)(\log{N}+\gamma)+L'(1,\chi)+O\left(|D|^{1/4}N^{-1/2}\log{2N}\right)
\end{equation*}
which formula implies the following bound
\begin{equation}
  \label{eq:1.9}
  \sum_{Q^4<n\le N}\lambda_0(n)n^{-1}\le L(1,\chi)\log{N}
\end{equation}
provided $|D|=(2\pi Q)^2$ is sufficiently large.

\begin{definition}
  We say that the character $\chi\pmod{|D|}$ is ``$\varepsilon$-exceptional'' if
  \begin{equation}
    \label{eq:1.10}
    L(1,\chi)\log{|D|}\le\varepsilon.
  \end{equation}
  Then, the corresponding discriminant $D$ of the field $K=\Q(\sqrt{D})$ is called ``$\varepsilon$-exceptional''.
\end{definition}

By \eqref{eq:1.9} and \eqref{eq:1.10} one infers \eqref{eq:1.5}. Therefore, if the discriminant $D$ is ``$\varepsilon$-exceptional'', then every $L(s,\psi)$ is $\varepsilon$-lacunary.

We recall some worthy shortcut notations which are used in analytic number theory. First, if $f,g$ are complex-valued functions, then the relation $f\ll g$ means that $|f|\le cg$ holds for all the relevant arguments with certain (implied) constant $c>0$. Next, the relation $f\asymp g$ means that $f\ll g$ and $g\ll f$ hold. Then $O(g)$ stands for a function (or a quantity) which is $\ll g$. Note that the above relations can hold only if $g\ge 0$. For example the statement $\sin{x}\ll\sin{x}$ is false. If $t,T$ are real numbers, then the notation $t\sim T$ stands for the inequality $T<t\le 2T$. Occasionally, we shall use the same symbol to denote different things, but the reader should not be confused, because the proper meaning will be clear from the context.

\section{Statement of Results}
\label{section:StatementOfResults}

Let $N(T)$ denote the number of zeros $\rho=\beta+i\gamma$ of $L(s)$ (counted with multiplicity) in the rectangle $0<\beta<1$, $\gamma\sim T$. By contour integration using the functional equation one derives the formula
\begin{equation}
  \label{eq:2.1}
  N(T)=\frac{T}{\pi}\log{QT}+O(T)
\end{equation}
for any $T\ge Q$. Here the dominant term emerges from variation of the argument of the root factor and the implied constant depends only on the gamma parameters.

Let $N_0(T)$ denote the number of zeros $\rho=1/2+i\gamma$ of $L(s)$ (counted with multiplicity) in the segment $\gamma\sim T$. Any natural $L$-function should satisfy the Riemann Hypothesis so $N_0(T)$ should be equal to $N(T)$. For $L$-functions of degree $1$ or $2$ it is known that a positive proportion of zeros are on the critical line $\Re s=1/2$, that is $N_0(T)\asymp N(T)$ for all $T$ with $\log{T}$ sufficiently larger than $\log{Q}$. It seems possible to show that if $L(s)$ is $\varepsilon$-lacunary (see the condition \eqref{eq:1.5}), then
\begin{equation*}
  N_0(T)=\{1+O(\varepsilon^{1/2})\}N(T)+O(T)
\end{equation*}
for all $T$ with $Q^{1/\sqrt{\varepsilon}}\le T\le Q^{1/\varepsilon}$.

For transparency we work out only the case of functions $L(s)=L(s,\psi)$ which are attached to the characters $\psi$ on ideal classes of the quadratic field $K=\Q(\sqrt{D})$. In greater generality as described in Section \ref{section:introduction} the arguments should be very similar. Our main result is the following

\begin{theorem}\label{thm:2.1}
  Let $N_{00}(T)$ denote the number of simple zeros $\rho=1/2+i\gamma$ of $L(s,\psi)$ with $\gamma\sim T$ and $N(T)$ the number of all zeros $\rho=\beta+i\gamma$ of $L(s,\psi)$ with $0<\beta<1$, $\gamma\sim T$ counted with multiplicity. We have
  \begin{equation}
    \label{eq:2.2}
    N_{00}(T)=N(T)+O\left(T\log{|D|}+(L(1,\chi)\log{T})^{1/4}T\log{T}\right)
  \end{equation}
  where the implied constant is absolute. Putting $\varepsilon=\varepsilon(D)=L(1,\chi)\log{|D|}$ we get
  \begin{equation}
    \label{eq:2.3}
    N_{00}(T)=\{1+O(\delta)\}N(T)
  \end{equation}
  with any $\delta\ge\varepsilon^{1/5}$ and every $T$ with $|D|^{1/\delta}\le T\le |D|^{\delta^4/\varepsilon}$, where the implied constant is absolute.
\end{theorem}

\begin{remarks*}
  First of all, the approximate formulas \eqref{eq:2.2} and \eqref{eq:2.3} are unconditional, but of course, \eqref{eq:2.3} is meaningful only if $\varepsilon=\varepsilon(D)$ is sufficiently small. This does not hold in reality since the Riemann Hypothesis implies the lower bound $L(1,\chi)\gg 1/\log\log{|D|}$. But so far the best known bound is $L(1,\chi)\gg |D|^{-\theta}$ with any $\theta>0$, the result due to C.L. Siegel which is not effective (the implied constant depends on $\theta$ and it cannot be computed numerically if $\theta<1/2$). Therefore, it is still interesting to speculate on the effect of the assumption that $\varepsilon(D)\to 0$ as $D$ varies over some infinite sequence of discriminants no matter how sparse it is.
\end{remarks*}

\begin{definition}
  An infinite sequence of discriminants $D$ is called ``exceptional'' if
  \begin{equation}
    \label{eq:2.4}
    \varepsilon(D)=L(1,\chi)\log{|D|}\to 0.
  \end{equation}
\end{definition}

\begin{corollary}
  As $\varepsilon=\varepsilon(D)\to 0$ over an exceptional sequence, then for every $\psi\in\widehat\Cl(K)$ the critical simple zeros of $L(s,\psi)$ of height $\sim T$ comprise a 100\% of all the zeros of height $\sim T$ for any $T$ with $|D|^{-\log\varepsilon}\le T\le |D|^{-1/\varepsilon\log\varepsilon}$.
\end{corollary}

This result sounds more impressive when applied for the trivial ideal class group character. In this case $L(s,\psi_0)$ factors into the Riemann zeta function $\zeta(s)$ and the Dirichlet $L$-function $L(s,\chi)$. For each factor separately we know the true values (asymptotically correct estimates) for the full numbers of zeros (counted with multiplicity) in the rectangle $s=\sigma+it$ with $0<\sigma<1$, $t\sim T$; these are 
\begin{equation*}
  \frac{T}{2\pi}\log{T}+O(T), \hspace{0.5cm} \frac{T}{2\pi}\log{|D|T}+O(T)
\end{equation*}
respectively. The sum of these values agrees with $N(T)$ (see \eqref{eq:2.1}), consequently Theorem \ref{thm:2.1} implies

\begin{corollary}\label{cor:2.3}
  Let $D$ run over an exceptional sequence of discriminants so $\varepsilon=\varepsilon(D)=L(1,\chi)\log{|D|}\to 0$. Choose any $T$ with $|D|^{-\log\varepsilon}\le T\le |D|^{-1/\varepsilon\log\varepsilon}$. Then the number of critical zeros of $\zeta(s)$ of height $\sim T$ which are simple and different from these of $L(s,\chi)$ approaches asymptotically the number of all zeros of $\zeta(s)$ of height $\sim T$.
\end{corollary}

If $L(s)$ is a lacunary $L$-function of degree two, then the twisted $L$-function
\[L(s;\lambda\chi')=\sum_{n=1}^\infty \lambda(n)\chi'(n)n^{-s}\]
by any fixed Dirichlet character $\chi'$ is also lacunary of degree two. The arguments presented in this paper for $\chi'=1$ and $\lambda=1\star\chi$, where $\chi$ is the real character to exceptional conductor $|D|$, extend easily to $L(s;\lambda\chi')=L(s,\chi')L(s,\chi\chi')$.
In particular Corollary \ref{cor:2.3} generalizes to any Dirichlet $L$-function $L(s,\chi')$ in place of $\zeta(s)$.

If one is willing to assume that
\[L(1,\chi)\ll(\log{|D|})^{-2015}\]
for an infinite sequence of discriminants $D$, then the same results would be achieved much faster by substantially simpler arguments (ignoring the sifting effects of the mollifier in various places). However the arguments are not powerful enough to cover the $L$-functions of degree larger than two, even if the lacunarity condition is assumed to be extremely strong.
\section{Levinson's Method}
\label{section:LevinsonsMethod}

There are two well established methods for counting zeros of $L$-functions on the critical line -- the Selberg method \cite{S} and the Levinson method \cite{L}. They are diametrically opposite to each other. Selberg's method relays on observing the sign changes of a suitably normalized and mollified $L$-function as its argument runs over a segment of the critical line. There is no risk of getting negative bound for the counting number, but the method is not perfect for intrinsic reasons; for one that the zeros are not supposed to be almost evenly spaced. Therefore, it needs a sensitive design for counting the zeros adequately (asymptotically precise) when passing through the segments. Yet, it may be the case that under the lacunarity condition the zeros do pretend to be more or less evenly spaced, contrary to the Pair Correlation Conjecture of Montgomery. This question was addressed by R. Heath-Brown during the AIM conference in Seattle of August 1996 (unpublished).

The method of Levinson is risky, because it may produce a negative bound for the counting number of critical zeros if the relevant estimates are crude. On the other hand it opens a possibility for accounting a 100\% of the critical zeros if the mollification is nearly perfect. This is indeed the scenario for lacunary $L$-functions. A far reaching version of Levinson's method has been developed in \cite{C}, see also the Appendix in \cite{CIS}. In this section we are going to adopt Proposition A of \cite{CIS} to our particular context. We shall also borrow numerous arguments developed in \cite{CI} for handling the off-diagonal terms in Sections \ref{section:introductionOffDiagonalTerms}-\ref{section:backOffDiagonalTerms}.

Thanks to the lacunarity of $L(s)$ we do not care about delicate
choices of the parameters involved in Levinson's original setup. We
shall also take numerous advantages of the lacunarity for technical
simplifications. In particular the root factor does not play a role
(no hassling with cross-terms). Although, the off-diagonal terms do
appear, their contribution is nominal, because the lacunarity strikes
twice independently. However, by no means one can neglect the
off-diagonal terms quickly. For simplicity we shall sacrifice some
surplus of the gain, but of course, not everything (see, for example,
how we derived the bound \eqref{eq:22.9} from the expression \eqref{eq:22.6}).

We start with the linear combination of $L(s)$ and its derivative:
\begin{equation}
  \label{eq:3.2}
  G(s)=L(s)+L'(s)/\log{N}
\end{equation}
with $N\ge 2$ (the level) to be chosen later, see \eqref{eq:4.14}. To $G(s)$ we attach a mollifier which is given by a Dirichlet polynomial
\begin{equation}
  \label{eq:3.3}
  M(s)=\sum_{m\le M}v(m)m^{-s}
\end{equation}
with coefficients $v(m)$ to be determined later subject to $v(1)=1$, $|v(m)|\le\tau(m)$. For now we assume that the mollifier $M(s)$ has length $M\le T^{1/2}$, but we shall see that shorter mollifiers do their designated job (to produce sifting effects) pretty well, again due to the lacunarity properties. Putting
\begin{equation}
  \label{eq:3.6}
  F(s)=G(s)M(s)-1
\end{equation}
we have Levinson's inequality (see Proposition A of \cite{CIS})
\begin{eqnarray}
  \label{eq:3.7}
  N_{00}(T)&\ge&N(T)-\frac{1}{\pi a}\int_T^{2T}\log{\left|1+F\left(\frac{1}{2}-a+it\right)\right|}dt+O(T)\\
  &\ge&N(T)-\frac{1}{\pi a}I_a(T)+O(T)\nonumber
\end{eqnarray}
where
\begin{equation*}
  I_a(T)=\int_T^{2T}\left|F\left(\frac{1}{2}-a+it\right)\right|dt.
\end{equation*}
This holds for any $a>0$ and $T>Q^4$ with absolute implied constant. Since we do not care loosing an absolute constant factor it is possible to replace $I_a(T)$ by
\begin{equation}
  \label{eq:3.8}
  I(T)=\int_T^{2T}\left|F\left(\frac{1}{2}+it\right)\right|dt.
\end{equation}

\begin{lemma}
  Let $T\ge M^2\ge Q^8$. For $0<a\le 1/2$ we have
  \begin{equation}
    \label{eq:3.9}
    I_a(T)\le T^{4a}(I(T)+O(T^{7/8}))
  \end{equation}
  where the implied constant is absolute.
\end{lemma}
\begin{proof}
  Put
  \[H(z)=\frac{2a(4-a^2)}{(z^2-a^2)(z^4-4)}T^{4z}.\]
  It is clear that
  \begin{eqnarray*}
    \frac{1}{2\pi}\int_{-\infty}^\infty |H(iv)|dv=\frac{a}{\pi}\int_{-\infty}^\infty \frac{4-a^2}{4+v^2}\frac{dv}{a^2+v^2}&<&1,\\[0.2cm]
    \frac{1}{2\pi}\int_{-\infty}^\infty |H(iv)|(v^2+1)dv&<&4,\\[0.2cm]
    \int_{-\infty}^\infty |H(-1+iv)|(v^2+1)dv&\ll&T^{-4}.
  \end{eqnarray*}
Since $H(z)$ has simple pole at $z=-a$ with residue $T^{-4a}$ we get
\[T^{-4a}F(s-a)=\frac{1}{2\pi i}\int_{(0)}F(s+z)H(z)dz-\frac{1}{2\pi i}\int_{(-1)}F(s+z)H(z)dz\]
if $\Re s=1/2$. Hence $T^{-4a}I_a(T)\le V-W$, say, where
\begin{eqnarray*}
  V&=&\frac{1}{2\pi}\int_{-\infty}^\infty |H(iv)|\int_T^{2T}\left|F\left(\frac{1}{2}+it+iv\right)\right|dtdv,\\
  W&=&\frac{1}{2\pi}\int_{-\infty}^\infty |H(-1+iv)|\int_T^{2T}\left|F\left(-\frac{1}{2}+it+iv\right)\right|dtdv.
\end{eqnarray*}
By the convexity bound for $L(s)$ we derive
\[F(s)\ll (MQ|s|)^{1/2}(\log{MQ|s|})^2,\hspace{0.5cm}\text{ if }\Re s=1/2.\]
Hence
\[\int_T^{T+v}\left|F\left(\frac{1}{2}+it\right)\right|dt\ll (v^2+1)(MQT)^{1/2}(\log{T})^2.\]
The same bound holds with $T$ replaced by $2T$. Hence we get
\[V\le\int_T^{2T}\left|F\left(\frac{1}{2}+it\right)\right|dt+O(T^{7/8}).\]
Next, by the functional equation for $L(s)$ and the trivial estimation on the line $\Re s=3/2$ we derive
\[F(s)\ll (MQ|s|)^2,\hspace{0.5cm}\text{ if }\Re s=-1/2.\]
Hence
\[W\ll (MQT)^2\int |H(-1+iv)|(v^2+1)dv\ll 1.\]
This completes the proof of \eqref{eq:3.9}.
\end{proof}

After inserting \eqref{eq:3.9} into \eqref{eq:3.7} it is clear that the best choice of the shift is $a=1/4\log{T}$ giving

\begin{lemma}
  Let $T\ge M^2\ge Q^8$. Then
  \begin{equation}
    \label{eq:3.10}
    N_{00}(T)>N(T)-4I(T)\log{T}+O(T)
  \end{equation}
  where $I(T)$ is given by \eqref{eq:3.8} with $F(s)=G(s)M(s)-1$ and the implied constant is absolute.
\end{lemma}

We are going to show that $I(T)/T$ is small, provided the mollifying factor $M(s)$ is chosen properly. Naturally this suggests that $M(s)$ should pretend to be the inverse of $L(s)$, or slightly better of $G(s)$, but due to the lacunarity it does not matter which one is on the target. Writing
\begin{equation}
  \label{eq:3.11}
  L(s)^{-1}=\prod_p (1-\lambda(p)p^{-s}+\chi(p)p^{-2s})=\sum_m \rho(m)m^{-s},
\end{equation}
we get the multiplicative function $\rho(m)$ with 
\begin{equation}
  \label{eq:3.12}
  \rho(p)=-\lambda(p), \ \rho(p^2)=\chi(p), \ \rho(p^\alpha)=0\text{ if }\alpha>2.
\end{equation}
We take
\begin{equation}
  \label{eq:3.13}
  M(s)=\sum_{m\le M}\rho(m)g(m)m^{-s}
\end{equation}
where $g(m)$ is a nice cropping function. For instance
\begin{equation}
  \label{eq:3.14}
  g(m)=\left(1-\frac{\log m}{\log M}\right)^r
\end{equation}
with $r$ a sufficiently large integer will do the job ($r=32$ is fine). The large degree of vanishing at the end point $m=M$ is necessary for our technique of producing some sifting effects.

Note that $\rho(m)$ is supported on cubefree numbers,
\begin{align}
  \label{eq:3.15}
  &|\rho(m)|\le\lambda_0(m)&&\hspace{-1.7cm}\text{ for all }m\\
  &\rho(m)=\mu(m)\lambda(m)&&\hspace{-1.7cm}\text{ if }m\text{ is squarefree.}
\end{align}

\begin{remarks*}
  Certain parts of the forthcoming sums are supported on the mollifier terms $m\equiv 0\pmod{|D|}$, specifically the off-diagonal constituents \eqref{eq:16.8}. We could easily eliminate these parts right now by restricting \eqref{eq:3.13} to $m\not\equiv 0\pmod{|D|}$. This incomplete mollifier does the job as good as the full one, because it is easy to estimate the missing terms by $|D|^{-1/2}T(\log{T})^{2015}$ directly using Cauchy-Schwarz inequality and the mean value estimates for $|G(s)|^2$ and $|M(s)|^2$ (giving up the lacunarity features and the sifting effects). This alteration can be implemented any time so we postpone the issue to the comments in the last section.
\end{remarks*}

\section{A Partition of $G(s)$}
\label{section:partitionG(s)}

To apply the mollifier $M(s)$ to $G(s)$ and observe its sifting effects we need to expand $G(s)$ into Dirichlet polynomials. To this end we fix two smooth functions $a(x),b(x)$ on $\R$ with 
\begin{equation}
  \label{eq:4.1}
  a(x)+b(x)=1-x,
\end{equation}
$a(x)$ supported on $x\le \alpha$ and $b(x)$ supported on $x\ge \beta$, where $0<\beta<\alpha<1$ are fixed numbers (see Figure \ref{fig:fig1}). In applications we shall choose $\alpha,\beta$ greater but close to $\frac{1}{2}$.

\begin{figure}[h!]
  \centering
  \def\svgwidth{160pt}
  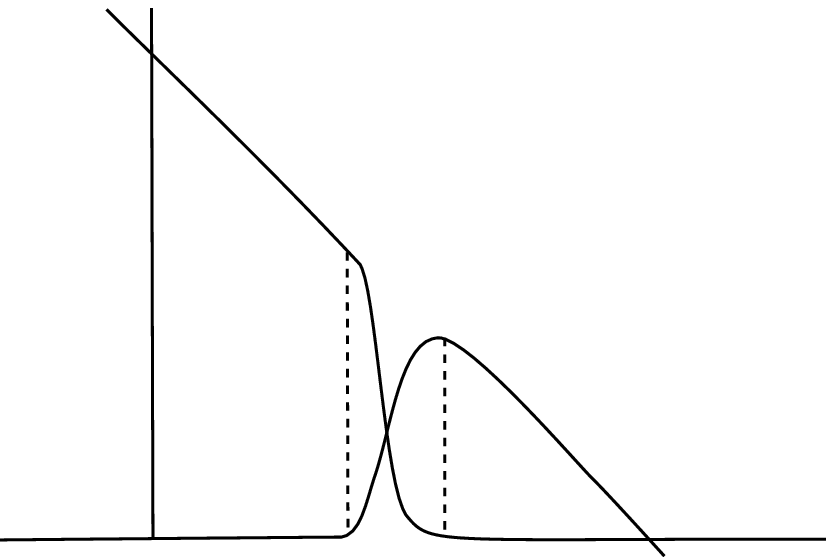
  \caption{}
  \label{fig:fig1}
\end{figure}

Given $N\ge 2$ consider the function
\begin{equation}
  \label{eq:4.2}
  f(z)=z^2\int_0^\infty a\left(\frac{\log y}{\log N}\right)y^{z-1}dy.
\end{equation}
The integral converges absolutely in $\Re z>0$. Integrating by parts we get
\begin{equation}
  \label{eq:4.3}
  f(z)\log N=\int_\beta^\alpha a''(u)N^{uz}du.
\end{equation}
This expression shows that $f(z)$ is an entire function and its power series expansion begins with
\begin{equation}
  \label{eq:4.4}
  f(z)=(\log{N})^{-1}+z+\dots
\end{equation}
By Mellin inversion, \eqref{eq:4.2} yields
\begin{equation}
  \label{eq:4.5}
  a\left(\frac{\log{y}}{\log{N}}\right)=\frac{1}{2\pi i}\int_{(1)}f(z)y^{-z}z^{-2}dz, \ \text{ if }y>0.
\end{equation}
Moving to the line $\Re z=-1$ we get by \eqref{eq:4.4} and \eqref{eq:4.1}
\begin{equation}
  \label{eq:4.6}
  b\left(\frac{\log{y}}{\log{N}}\right)=\frac{-1}{2\pi i}\int_{(-1)}f(z)y^{-z}z^{-2}dz, \ \text{ if }y>0.
\end{equation}
Now consider the Dirichlet polynomial
\begin{equation}
  \label{eq:4.7}
  A(s)=\sum_n a\left(\frac{\log{n}}{\log{N}}\right)\lambda(n)n^{-s}
\end{equation}
which has length $N^\alpha$. Let $s$ be in the critical strip $0<\Re s<1$. By \eqref{eq:4.5} we derive
\begin{eqnarray*}
  A(s)&=&\frac{1}{2\pi i}\int_{(1)}L(s+z)f(z)z^{-2}dz\\
  &=&\frac{1}{2\pi i}\int_{(-1)}L(s+z)f(z)z^{-2}dz+G(s)+Rf(1-s)(1-s)^{-2}
\end{eqnarray*}
where the second term $G(s)=L(s)+L'(s)/\log{N}$ comes from the double pole at $z=0$ and the third term comes from a possible simple pole at $z=1-s$. In the integral over the line $\Re z=-1$ we apply the functional equation $L(s+z)=X(s+z)\overline L(1-s-z)$ with 
\begin{equation}
  \label{eq:4.8}
  X(s+z)=X(s)(Q|s|)^{-2z}\{1+z^2\eta(s,z)\},
\end{equation}
say, see \eqref{eq:1.3} and \eqref{eq:1.4}. This integral splits accordingly
\[\frac{1}{2\pi i}\int_{(-1)}L(s+z)f(z)z^{-2}dz=-X(s)B(s)-X(s)R(s)\]
where
\begin{equation}
  \label{eq:4.9}
  B(s)=\frac{-1}{2\pi i}\int_{(-1)}\overline L(1-s-z)(Q|s|)^{-2z}f(z)z^{-2}dz
\end{equation}
and
\begin{equation}
  \label{eq:4.10}
  R(s)=\frac{1}{2\pi i}\int_{(-\varepsilon)}\overline L(1-s-z)(Q|s|)^{-2z}f(z)\eta(s,z)dz.
\end{equation}
Note that in $R(s)$ we moved back the integration to the line $\Re
z=-\varepsilon$ with $\varepsilon=1/\log Q|s|$ without passing poles, because $\eta(s,z)$ is holomorphic in $z$ (not in $s$). Now, by \eqref{eq:4.6}, the integral \eqref{eq:4.9} expands into the series
\begin{equation}
  \label{eq:4.11}
  B(s)=\sum_n b\left(\frac{\log Q^2|s|^2/n}{\log{N}}\right)\overline\lambda(n)n^{s-1}.
\end{equation}
We have proved the following formula
\begin{proposition}\label{prop:4.1}
  For $s$ in the critical strip $0<\Re s<1$ we have
  \begin{equation}
    \label{eq:4.12}
    G(s)=A(s)+X(s)B(s)+X(s)R(s)+Rf(1-s)(1-s)^{-2}
  \end{equation}
  where $A(s),B(s)$ are given by \eqref{eq:4.7}, \eqref{eq:4.11}, respectively, $R(s)$ is the integral \eqref{eq:4.10} and $R$ denotes the residue of the $L$-function; $R=L(1,\chi)$ if $\psi$ is trivial and $R=0$ otherwise.
\end{proposition}

If $\Re s=1/2$, then $R(s)$ can be easily estimated by 
\begin{equation}
  \label{eq:4.13}
  R(s)\ll Q^{1/2}|s|^{-\frac{1}{12}}.
\end{equation}
This follows by $L(1-s-z)\ll(|s+z|Q)^{1/2}$ (the convexity bound),
$f(z)\log N\ll 1$ (see~\eqref{eq:4.3}) and $\eta(s,z)\ll
|z|^{-1}(|s|+|z|)^{-1}$ (see~\eqref{eq:1.4}).

Assume $T\ge Q^8$ and choose the level of Levinson's form \eqref{eq:3.2}
\begin{equation}
  \label{eq:4.14}
  N=Q^2T^2.
\end{equation}
From now on we let $s$ be in the segment
\begin{equation}
  \label{eq:4.15}
  s=\frac{1}{2}+it, \ T\le t\le 2T.
\end{equation}
Then the root factor satisfies $|X(s)|=1$ and for the polar term we derive by \eqref{eq:4.3}
\[f(1-s)(1-s)^{-2}\ll N^{\alpha/2}T^{-2}\le N^{1/2}T^{-2}=QT^{-1}.\]
The residue of $L(s)$ is $R\ll\log{Q}$. Therefore, Proposition \ref{prop:4.1} yields
\begin{corollary}\label{cor:4.2}
  For $s$ in the segment \eqref{eq:4.15} we have
  \begin{equation}
    \label{eq:4.16}
    G(s)=A(s)+X(s)B(s)+O(T^{-1/48})
  \end{equation}
  where the implied constant is absolute.
\end{corollary}

Next we look into coefficients of $B(s)$, they depend on $|s|$ mildly if $s$ is in the segment \eqref{eq:4.15}. Precisely we have
\[b\left(\frac{\log{Q^2|s|^2/n}}{\log{N}}\right)=b\left(1-\frac{\log{n}}{\log{N}}+\delta\right)\]
where
\begin{equation}
  \label{eq:4.17}
  \delta=\delta(s)=\frac{2\log{|s|/T}}{\log{N}}, \ 0\le\delta\le \frac{\log{4}}{\log{N}}.
\end{equation}

The small shift in the function $b(1-x+\delta)$ is dispensable, it can be isolated by Taylor's expansion
\[b(1-x+\delta)=b(1-x)+\delta b'(1-x)+\frac{1}{2}\delta^2b''(1-x)+\Delta(s,x),\]
say, with the last term $\Delta(s,x)$ being sufficiently small for
easy direct handling (the sifting effect can be ignored). Nevertheless
we opt working with one function $b(1-x+\delta)$ contaminated by the shift $\delta$ rather than with several derivatives without the shift.

For $s$ on the line $\Re s=1/2$ we can write the complex conjugate of $B(s)$, see \eqref{eq:4.11}, in the following fashion
\begin{equation}
  \label{eq:4.18}
  \overline B(s)=\sum_n b^*\left(\frac{\log{n}}{\log{N}}\right)\lambda(n)n^{-s}
\end{equation}
where
\begin{equation}
  \label{eq:4.19}
  b^*(x)=b(1-x+\delta).
\end{equation}
Note that $b^*(x)$ is supported on $x\le 1-\beta+\delta$ so $B(s)$ runs over $1\le n\le 4N^{1-\beta}$. We have
\begin{equation}
  \label{eq:4.20}
  b^*(x)=x-\delta, \ \text{ if }x\le 1-\alpha+\delta,
\end{equation}
so the sum \eqref{eq:4.18} begins with
\begin{equation}
  \label{eq:4.21}
  \sum_{n\le N^{1-\alpha}}\left(\frac{\log{n}}{\log{N}}-\delta\right)\lambda(n)n^{-s}.
\end{equation}
On the other hand the sum \eqref{eq:4.7} begins with
\begin{equation}
  \label{eq:4.22}
  \sum_{n\le N^\beta}\left(1-\frac{\log{n}}{\log{N}}\right)\lambda(n)n^{-s}.
\end{equation}

\section{Estimating $I(s)$}
\label{section:estimatingI(s)}

Multiplying \eqref{eq:4.16} by the mollifier \eqref{eq:3.13} we obtain the inequality
\begin{equation}
  \label{eq:5.1}
  |F(s)|\le|C(s)|+|C^*(s)|+O(|M(s)|T^{-1/48})
\end{equation}
with $C(s)=A(s)M(s)-1$ and $C^*(s)=\overline B(s)M(s)$ for $s$ in the segment \eqref{eq:4.15}. The Dirichlet polynomials
\[C(s)=\sum_{1<\ell<MN^\alpha} c(\ell)\ell^{-s},\ C^*(s)=\sum_{1\le \ell<4MN^{1-\beta}} c^*(\ell)\ell^{-s}\]
have coefficients given by convolutions; specifically
\begin{eqnarray}\label{eq:5.2}
  c(\ell)=\sum_{mn=\ell}\rho(m)\lambda(n)g(m)h(n), \ && h(n)=a\left(\frac{\log{n}}{\log{N}}\right),\\
  \label{eq:5.3}
  c^*(\ell)=\sum_{mn=\ell}\rho(m)\lambda(n)g(m)h^*(n), \ && h^*(n)=b^*\left(\frac{\log{n}}{\log{N}}\right),
\end{eqnarray}
and $g(m)$ is the cropping function which we have chosen in \eqref{eq:3.14}. Inserting the inequality \eqref{eq:5.1} into the integral \eqref{eq:3.8} we obtain
\begin{equation}
  \label{eq:5.4}
  I(T)\le J(T)+J^*(T)+O(T^{47/48}(\log{T})^2)
\end{equation}
where
\[J(T)=\int_T^{2T}\left|C\left(\frac{1}{2}+it\right)\right|dt, \hspace{0.5cm} J^*(T)=\int_T^{2T}\left|C^*\left(\frac{1}{2}+it\right)\right|dt\]
and the error term comes from the classical mean-value theorem for Dirichlet polynomials.

Next we apply the Cauchy-Schwarz inequality and we expand the range of integration getting $J(T)^2\le T\Kc(T)$ where
\begin{equation}
  \label{eq:5.5}
  \Kc(T)=\int \Phi\left(\frac{t}{T}\right)\left|C\left(\frac{1}{2}+it\right)\right|^2dt
\end{equation}
and $\Phi(u)$ is any non-negative smooth function, compactly supported on $\R^+$ with $\Phi(u)\ge 1$ in the interval $1\le u\le 2$. Similarly $J^*(T)^2\le T\Kc^*(T)$ where
\begin{equation}
  \label{eq:5.6}
  \Kc^*(T)=\int_T^{2T}\left|C^*\left(\frac{1}{2}+it\right)\right|^2dt,
\end{equation}
but without smoothing, because it would give no advantage. Therefore we have shown the following inequality
\begin{equation}
  \label{eq:5.7}
  I(T)\le (T\Kc(T))^{\frac{1}{2}}+(T\Kc^*(T))^{\frac{1}{2}}+O(T^{47/48}(\log{T})^2).
\end{equation}

\section{The Diagonals}
\label{section:Diagonals}

We take $\beta$ somewhat larger than $\frac{1}{2}$ to make sure that the sum $C^*(s)$ is shorter than $T$. Specifically $C^*(s)$ has length $\le 4MN^{1-\beta}$ so it is enough to assume that
\begin{equation}
  \label{eq:6.1}
  MN^{1-\beta}\le T(\log{T})^{-11}.
\end{equation}
Then the classical mean-value theorem for Dirichlet polynomials shows that $\Kc^*(T)$ is equal to the contribution of the diagonal terms up to a small error term. This claim requires some explanation, because the coefficients of $C^*(s)$ depend on $|s|$. We have
\[\Kc^*(T)=\sum_{\ell_1}\sum_{\ell_2} (\ell_1\ell_2)^{-\frac{1}{2}}\int_T^{2T}c^*(\ell_1)c^*(\ell_2)(\ell_1/\ell_2)^{it}dt\]
and $c^*(\ell_1)c^*(\ell_2)$ is equal to
\[\mathop{\sum\sum}_{\substack{m_1n_1=\ell_1\\ m_2n_2=\ell_2}} \rho(m_1)\rho(m_2)\lambda(n_1)\lambda(n_2)g(m_1)g(m_2)b\left(s;\frac{\log{n_1}}{\log{N}},\frac{\log{n_2}}{\log{N}}\right)\]
with $b(s;x_1,x_2)=b(\delta(s)+1-x_1)b(\delta(s)+1-x_2)$ and $\delta(s)$ given by \eqref{eq:4.17}. The diagonal terms $\ell_1=\ell_2$ contribute $T(\Kc_0^*+c^*(1)^2)$, where
\begin{equation}
  \label{eq:6.2}
  \Kc_0^*=\sum_{1< \ell<4MN^{1-\beta}} c^*(\ell)^2\ell^{-1}
\end{equation}
and $c^*(\ell)$ are taken with the $\delta=\delta(s)$ at some fixed point $s$ in the segment \eqref{eq:4.15} by the mean-value theorem. Note we have extracted the first term $c^*(1)=h^*(1)=b^*(0)=-\delta\ll 1/\log{N}$.

If $\ell_1\not=\ell_2$ we integrate by parts. Since
\[\frac{\partial}{\partial t}b(s;x_1,x_2)\ll \frac{\partial}{\partial t}\delta(s)\asymp t/(t^2+1)\log{N}\asymp 1/T\log{N},\]
we estimate the contribution of the off-diagonal terms by
\[\mathop{\sum\sum}_{\ell_1\not=\ell_2<4MN^{1-\beta}} \tau^2(\ell_1)\tau^2(\ell_2)(\ell_1\ell_2)^{-\frac{1}{2}}|\log{\ell_1/\ell_2}|^{-1}\ll MN^{1-\beta}(\log{MN})^{10}\]
(the small shift $\delta(s)$ does not influence the bound). By our assumption \eqref{eq:6.1} we conclude that
\begin{equation}
  \label{eq:6.3}
  \Kc^*(T)=T\Kc_0^*+O(T/\log{T}).
\end{equation}

In case of $\Kc(T)$ there is no perturbation by the shift $\delta(s)$ so the diagonal contribution is exactly $\hat\Phi(0)T\Kc_0$, where
\begin{equation}
  \label{eq:6.4}
  \Kc_0=\sum_{1<\ell<MN^\alpha} c(\ell)^2\ell^{-1}.
\end{equation}
Note there is no term with $l=1$. However the polynomial $C(s)$ has length $MN^\alpha>MN^\beta>MN^{1/2}=MQT>T$ so there is a significant contribution of the off-diagonal terms, say $T\Kc^{\neq}(T)$, where
\begin{equation}
  \label{eq:6.5}
  \Kc^\neq(T)=\mathop{\sum\sum}_{1<\ell_1\neq \ell_2<MN^\alpha}\Psi(T\log{\ell_1/\ell_2})c(\ell_1)c(\ell_2)(\ell_1\ell_2)^{-\frac{1}{2}}
\end{equation}
and $\Psi(2\pi v)$ stands for the Fourier transform of $\Phi(u)$. We have
\begin{equation}
  \label{eq:6.6}
  \Kc(T)=\Psi(0)T\Kc_0+T\Kc^\neq(T).
\end{equation}

Our goal is to show that the diagonal sums $\Kc_0$ and $\Kc_0^*$ are small. The partial sums
\begin{equation}
  \label{eq:6.7}
  S(X,Y)=\sum_{X<\ell\le Y} c(\ell)^2 \ell^{-1}
\end{equation}
with $1\le X< Y$ need different handling in various ranges. The corresponding partial sums
\begin{equation}
  \label{eq:6.8}
  S^*(X,Y)=\sum_{X<\ell\le Y} c^*(\ell)^2\ell^{-1}
\end{equation}
are similar to $S(X,Y)$ so we shall treat $S(X,Y)$ in details and only occasionally we shall make comments to illuminate small differences. The final estimates for $S(X,Y)$ and $S^*(X,Y)$ will be the same.

The off-diagonal sum $\Kc^\neq(T)$ requires a lot more sophisticated analysis which we postpone to the last ten sections.

\section{Estimation of the First Diagonal Terms}
\label{section:estimationFirstDiagonalTerms}

The lacunarity of $\lambda(\ell)$ with small $\ell$ is not frequent and the best available bound is the trivial one $|\lambda(\ell)|\le\tau(n)$. However every convolution coefficient $c(\ell), \ell\not=0$, is small by the sifting effect. We shall see that in the range $\ell\le M$ the coefficients $c(\ell)$ are supported on almost primes ($c(\ell)$ vanishes if $\ell$ has many distinct prime divisors).

In this section only we introduce the von Mangoldt functions $\Lambda_j(n)$ of degree $j=0,1,2,\dots,r$, which are derived from the Euler product \eqref{eq:1.1} and scaled down by factors $(\log{M})^{-j}$. Sorry we use the same notation as for the classical von Mangoldt functions derived from $\zeta(s)$, hopefully without confusion.

If $\ell\le M$, then \eqref{eq:5.2} becomes
\begin{equation}
  \label{eq:7.1}
  c(\ell)=\sum_{mn=\ell} \rho(m)\lambda(n)\left(1-\frac{\log{m}}{\log{M}}\right)^{r}\left(1-\frac{\log{n}}{\log{N}}\right)
\end{equation}
because the restrictions $m\le M$, $n\le N$ are redundant. Moreover \eqref{eq:5.3} becomes
\begin{equation}
  \label{eq:7.2}
  c^*(\ell)=\sum_{mn=\ell} \rho(m)\lambda(n)\left(1-\frac{\log{m}}{\log{M}}\right)^r\left(\frac{\log{n}}{\log{N}}-\delta\right).
\end{equation}
Clearly \eqref{eq:7.1} and \eqref{eq:7.2} are very similar. The generating Dirichlet series for the unrestricted convolution coefficients \eqref{eq:7.1} is equal to
\begin{eqnarray*}
  Z_r(s)&=&\left(\sum_n \lambda(n)\left(1-\frac{\log{n}}{\log{N}}\right)n^{-s}\right)\left(\sum_m \rho(m)\left(1-\frac{\log{m}}{\log{M}}\right)^r m^{-s}\right)\\
  &=&\left(L(s)+\frac{L'(s)}{\log{N}}\right)\left(\frac{M^s}{L(s)}\right)^{(r)} M^{-s}(\log{M})^{-r}
\end{eqnarray*}
and that for \eqref{eq:7.2} is equal to
\[Z_r^*(s)=-\left(\delta L(s)+\frac{L'(s)}{\log{N}}\right)\left(\frac{M^s}{L(s)}\right)^{(r)} M^{-s}(\log{M})^{-r}.\]
For example, for $r=0$ we get
\[Z_0(s)=1+\frac{L'(s)}{L(s)\log{N}}=1-\gamma\sum_\ell \Lambda_1(\ell)\ell^{-s}\]
with $\gamma=\log{M}/\log{N}$ (scaling adjustment factor). Therefore $c_1(\ell)=\gamma\Lambda_1(\ell)$ is supported on prime powers.

For any $r\ge 0$ we have the formula
\[\left(\frac{M^s}{L(s)}\right)^{(r)}=M^s(\log{M})^r\sum_{0\le j\le r}\binom{r}{j}(\log{M})^{-j}\left(\frac{1}{L(s)}\right)^{(j)}\]
and $(1/L(s))^{(j)}$ is equal to
\[\frac{j!}{L(s)}\sum_{a_1+2a_2+\dots=j}\frac{(a_1+a_2+\dots)!}{a_1!a_2!\dots}\left(\frac{-L'(s)}{1!L(s)}\right)^{a_1}\left(\frac{-L''(s)}{2!L(s)}\right)^{a_2}\cdots\]
Write $\binom{r}{j}=r!/j!b!$ with $b=r-j$ and $\log=\log{M}$. These formulas yield
\begin{eqnarray*}
  Z_r(s)=\left(1+\frac{\gamma L'}{L\log}\right)\sum_{b+a_1+2a_2+\dots=r}&& \frac{r!(a_1+a_2+\dots)!}{b!a_1!a_2!\dots}\\
  &&\left(\frac{-L'}{1!L\log}\right)^{a_1}\left(\frac{-L''}{2!L\log}\right)^{a_2}\cdots
\end{eqnarray*}
Comparing the coefficients in Dirichlet series expansions on both sides we conclude
\begin{lemma}
  \label{lemma:7.1}
  If $1\le\ell\le M$, then $c(\ell)$ can be written as the sum
  \[\sum_{a_1+2a_2+\dots\le r} \phi(a_1,a_2,\dots)(\Lambda_0-\gamma\Lambda_1)\star(\underbrace{\Lambda_1\star\dots\star\Lambda_1}_{a_1\text{ times}})\star (\underbrace{\Lambda_2\star\dots\star\Lambda_2}_{a_2\text{ times}})\star\dots\]
with suitable coefficients $\phi(a_1,a_2,\dots)$, $\phi(0,0,\dots)=1$.
\end{lemma}

Using the obvious estimate
\[\sum_{\ell\le Y}|\Lambda_1(\ell)|^2\ell^{-1}\ll \left(\frac{\log{Y}}{\log{M}}\right)^2\]
and recurrence formulas for $\Lambda_k$ we get
\[V_k(Y)=\sum_{\ell\le Y}|\Lambda_k(\ell)|^2\ell^{-1}\ll \left(\frac{\log{Y}}{\log{M}}\right)^{2k}.\]
Hence, putting $\Lambda_{(k)}=\Lambda_{k_1}\star\dots\star\Lambda_{k_n}$ for $(k)=(k_1,\dots,k_n)$ we derive
\[\sum_{\ell\le Y}|\Lambda_{(k)}(\ell)|^2\ell^{-1}\le V_{k_1}(Y)\dots V_{k_n}(Y)\ll \left(\frac{\log{Y}}{\log{M}}\right)^{2(k_1+\dots+k_n)}\]
where the implied constant depends only on $(k)$.

Since $c(\ell)$ for $1<\ell\le M$ is a linear combination of $\Lambda_k(\ell)$ with $(k)\not=(0,0,\dots)$ by Lemma \ref{lemma:7.1} we derive by the above estimates
\begin{lemma}
  \label{lemma:7.2}
  If $2\le Y\le M$, then
  \begin{equation}
    \label{eq:7.3}
    S(1,Y)\ll(\log{Y}/\log{M})^2
  \end{equation}
  where the implied constant depends only on $r$.
\end{lemma}

By the above arguments it is clear that the same bound \eqref{eq:7.3} holds for $S^*(1,Y)$.

The bound \eqref{eq:7.3} is valid for $Y\le M$, but it is good only if $Y$ is relatively smaller. Suppose $Q^8\le X<Y\le M$. Now the lacunarity of $\lambda(\ell)$ kicks in and we are going to exploit it on top of the mollifier sifting effects. Lemma \ref{lemma:7.1} shows that $c(\ell)$ vanishes if $\ell$ has more than $r$ distinct prime divisors. Writing uniquely $\ell=dk$, where $(d,k)=1$, $k$ squarefree, $d$ powerful, we get $c(\ell)\ll\lambda_0(k)\ll 1$ (recall that $\rho(m)$ is supported on cubefree numbers). Hence,
\begin{eqnarray*}
  S(X,Y)&\ll&\sum_{X<dk\le Y}\lambda_0(k)(dk)^{-1}\\[0.2cm]
  &\ll&\sum_d d^{-1}\sum_{\sqrt{X}<k\le Y} \lambda_0(k)k^{-1}+\sum_{d\ge\sqrt{X}}d^{-1}\sum_{k\le Y}\lambda_0(k)k^{-1}\\[0.2cm]
  &\ll&\sum_{\sqrt{X}<k\le Y}\lambda_0(k)k^{-1}+X^{-\frac{1}{4}}\log{Y}.
\end{eqnarray*}
Applying \eqref{eq:1.9} we derive the following estimate
\begin{lemma}
  \label{lemma:7.3}
  If $Q^8\le X<Y\le M$, then
  \begin{equation}
    \label{eq:7.4}
    S(X,Y)\ll L(1,\chi)\log{Y}
  \end{equation}
  where the implied constant depends only on $r$.
\end{lemma}

The same bound \eqref{eq:7.4} holds for $S^*(X,Y)$.

\section{Reducing to the Squarefree Diagonal Terms}
\label{section:reducingSquarefreeDiagonalTerms}

If $\ell>M$, then the convolution coefficients
\begin{equation}
  \label{eq:8.1}
  c(\ell)=\sum_{mn=\ell}\rho(m)\lambda(n)g(m)h(n)
\end{equation}
may not be supported on almost primes, because $g(m),h(n)$ are no longer polynomials in $\log{m}/\log{M}$ and $\log{n}/\log{N}$, respectively. Therefore the previous arguments fail. We shall estimate $S(X,Y)$ with $X\ge M$ in different ways. But first we reduce the sum $S(X,Y)$ to
\begin{equation}
  \label{eq:8.2}
  S^\flat(X,Y)=\sum_{\substack{X<\ell\le Y\\ (\ell,q)=1}} \mu(\ell)^2c(\ell)^2\ell^{-1}
\end{equation}
where $q$ is a fixed squarefree number to be chosen later.

Throughout $d$ runs over numbers such that $p\mid d\Rightarrow p^2\mid dq^2$. Writing uniquely $\ell=dk$ with $k$ squarefree, $(k,dq)=1$ we get
\[c(\ell)=\sum_{uv=d}\rho(u)\lambda(v)c_{uv}(k)\]
where
\begin{equation}
  \label{eq:8.3}
  c_{uv}(k)=\sum_{mn=k}\rho(m)\lambda(n)g(um)h(vn).
\end{equation}
Hence
\[c(\ell)^2\le\tau(d)^6\sum_{uv=d}c_{uv}(k)^2\]
and
\[S(X,Y)\le\sum_d \frac{\tau(d)^6}{d}\sum_{uv=d}S^\flat_{uv}\left(\frac{X}{d},\frac{Y}{d}\right)\]
where $S^\flat_{uv}(X,Y)$ stands for the sum \eqref{eq:8.2} with $c(\ell)$ replaced by $c_{uv}(\ell)$. The contribution of large $d$, say $d>U=(\log{Y})^{50}$, is negligible. Precisely, by trivial estimations, we derive the following:
\[|c(\ell)|\ll \tau_4(\ell),\]
\[S^\flat_{uv}\left(\frac{X}{d},\frac{Y}{d}\right)\ll(\log{Y})^{16},\]
\[\sum_{d>U}\frac{\tau(d)^6}{d}\sum_{uv=d} S^\flat_{uv}\left(\frac{X}{d},\frac{Y}{d}\right)\ll U^{-1}(\log{Y})^{48}.\]
Hence
\begin{equation}
  \label{eq:8.4}
  S(X,Y)\le \sum_{d\le U}\frac{\tau(d)^6}{d}\sum_{uv=d}S^\flat_{uv}\left(\frac{X}{d},\frac{Y}{d}\right)+O((\log{Y})^{-2}).
\end{equation}
The coefficients $c_{uv}(\ell)$ in $S^\flat_{uv}(X/d,Y/d)$ have slightly shifted crop functions;
\[g(um)=\left(1-\frac{\log{um}}{\log{M}}\right)^r=\left(1-\frac{\log{u}}{\log{M}}\right)^r \left(1-\frac{\log{m}}{\log{M/u}}\right)^r\]
if $1\le um\le M$, and
\[h(vn)=a \left(\frac{\log{vn}}{\log{N}}\right)=a \left(\delta+\frac{\log{n}}{\log{N}}\right)\text{ with }\delta=\frac{\log{v}}{\log{N}}.\]

In the next two sections we shall get estimates for $S^\flat(X,Y)$ which apply to every $S^\flat_{uv}(X/d,Y/d)$ with $uv=d\le U=(\log{Y})^{50}$. The small change of the crop functions ($g$ by rescaling $M\to M/u$ and $h$ by the shift $a(x)\to a(\delta+x)$) does not require any significant changes in the used arguments.

\section{Estimating $S^\flat(X,Y)$}
\label{section:estimatingSflat}

For $\ell=mn$ squarefree we have $\rho(m)\lambda(n)=\mu(m)\lambda(\ell)$ and $c(\ell)=\lambda(\ell)\theta(\ell)$ where $\theta=\mu g\star h$ is a kind of a sieve weight,
\begin{equation}
  \label{eq:9.1}
  \theta(\ell)=\sum_{m\mid \ell}\mu(m)g(m)h(\ell/m).
\end{equation}
This factorization separates the lacunarity feature of $\lambda(\ell)$ from the sifting feature of $\theta(\ell)$. We have $c(\ell)^2\le |\lambda(\ell)|^{\frac{1}{2}}|\theta(\ell)|\tau(\ell)^2$ and by Cauchy's inequality
\begin{equation}
  \label{eq:9.2}
  S^\flat(X,Y)\le \left(\sum_{X<\ell\le Y} |\lambda(\ell)|\ell^{-1}\right)^{\frac{1}{2}}\left(\sum_{X<\ell\le Y} \tau(\ell)^4\mu(q\ell)^2\theta(\ell)^2\ell^{-1}\right)^{\frac{1}{2}}.
\end{equation}
The first sum is bounded by $L(1,\chi)\log{Y}$ if $Y>X\ge Q^4$, see \eqref{eq:1.9}. The second sum is bounded by
\begin{equation}
  \label{eq:9.3}
  T(X,Y)=\sum_{X<\ell\le Y}\phi(\ell)\theta(\ell)^2
\end{equation}
where $\phi(\ell)$ is the completely multiplicative function such that
\begin{equation}
  \label{eq:9.4}
  \phi(p)=0\hspace{0.5cm}\text{ if }p\mid q, \hspace{1cm} \phi(p)=r/p\hspace{0.5cm}\text{ if }p\nmid q,
\end{equation}
with $r=16$. We shall estimate $T(X,Y)$ for any $r\ge 1$ which agrees with the exponent in the crop function \eqref{eq:3.14} of the mollifier \eqref{eq:3.13}. Our goal is the following estimate (we assume $q$ is divisible by every $p\le r^2$ so $\phi(p)< 1/\sqrt{p}$).
\begin{lemma}\label{lemma:9.1}
  If $Q^4\le M\le X<Y\le N$, then
  \begin{equation}
    \label{eq:9.5}
    T(X,Y)\ll(\log{Y}/\log{M})^r
  \end{equation}
where the implied constant depends on $r$.
\end{lemma}

By \eqref{eq:9.5} and \eqref{eq:9.2} we get

\begin{corollary}\label{ref:cor9.2}
  If $Q^4\le M\le X<Y\le N$, then
  \begin{equation}
    \label{eq:9.6}
    S^\flat(X,Y)\ll(L(1,\chi)\log{Y})^{\frac{1}{2}}(\log{Y}/\log{M})^{16}.
  \end{equation}
\end{corollary}
\begin{remarks*}
  If we assumed the stronger lacunarity property that $L(1,\chi)\ll(\log|D|)^{-r-6}$, then the trivial bound $T(X,Y)\ll(\log{Y})^{r+4}$ would have sufficed. However, we are willing to assume only that $\varepsilon(D)=L(1,\chi)\log|D|\to 0$, so our job is much harder.
\end{remarks*}

The same arguments work for $S_{uv}^\flat(X/d,Y/d)$ with $uv=d\le U=(\log{Y})^{50}$ giving the same bound \eqref{eq:9.5}. Hence \eqref{eq:8.4} yields.

\begin{corollary}\label{cor:9.3}
  If $Q^4\le M\le X<Y\le N$, then
  \begin{equation}
    \label{eq:9.7}
    S(X,Y)\ll (L(1,\chi)\log{Y})^{\frac{1}{2}}(\log{Y}/\log{M})^{16}+(\log{Y})^{-2}.
  \end{equation}
\end{corollary}

\section{Sums of the Möbius function}
\label{section:sumsMobius}

Typically for estimating sums involving the Möbius function one applies analytic methods by contour integration in the zero-free region of $\zeta(s)$. We opt more elementary path which goes through the Prime Number Theorem in the following form
\begin{equation}
  \label{eq:10.1}
  \sum_{m\ge X}\mu(m)m^{-1}\ll\exp(-c\sqrt{\log{x}})
\end{equation}
where $c$ is a positive constant. In the sequel $c$ stands for a positive constant different every time. By \eqref{eq:10.1} one derives
\begin{equation}
  \label{eq:10.2}
  \sum_{\substack{m\ge X\\ (m,k)=1}} \mu(m)m^{-1}\ll\sigma_{-1}(k)\exp(-c\sqrt{\log{x}}).
\end{equation}
Then \eqref{eq:10.2} yields the same bound for the sum twisted by the divisor functions
\begin{equation}
  \label{eq:10.3}
  \sum_{\substack{m\ge X\\ (m,k)=1}} \mu(m)\tau_r(m)m^{-1}\ll \sigma_{-1}(k)\exp(-c\sqrt{\log{x}}).
\end{equation}
Hence, if $f(x)$ is a function on $\R^+$ with $|f(x)|+x|f'(x)|\ll (\log{x})^A$ then we get (by partial summation)
\begin{equation}
  \label{eq:10.4}
  \sum_{\substack{m\ge X\\ (m,k)=1}} \mu(m)\tau_r(m)f(m)m^{-1}\ll\sigma_{-1}(k)\exp(-c\sqrt{\log{x}})
\end{equation}
where the implied constant depends on $A$. In particular
\begin{equation}
  \label{eq:10.5}
  \sum_{\substack{m\ge X\\ (m,k)=1}} \mu(m)\tau_r(m)(\log{m})^a\ll \sigma_{-1}(k)\exp(-c\sqrt{\log{x}}).
\end{equation}
Moreover, for the complete sum we have
\begin{equation}
  \label{eq:10.6}
  \sum_{(m,k)=1}\mu(m)\tau_r(m)(\log{m})^a=0\hspace{0.5cm}\text{ if }0\le a<r.
\end{equation}
Indeed, the complete sum \eqref{eq:10.6} is the $a$-th derivative (at $s=1$) of 
\[\sum_{(m,k)=1} \mu(m)\tau_r(m)m^{-s}=\prod_{p\nmid k} (1-r p^{-s})=\zeta(s)^{-r}\eta_{r k}(s)\]
say, where $\eta_{r k}(s)$ is holomorphic in $\Re{s}>1/2$. Since $\zeta(s)^{-r}$ has zero as $s=1$ of order $r$ the formula \eqref{eq:10.6} follows.

We shall also need the following formula
\begin{lemma}\label{lemma:10.1}
Let $\phi(\ell)$ be the completely multiplicative function defined by \eqref{eq:9.4}. Suppose every prime $p\le r^2$ divides $q$. Then
\begin{equation}
  \label{eq:10.7}
  \sum_{\ell\le X}\phi(\ell)=P_r(\log{x})+O(X^{-1/5r})
\end{equation}
where $P_r(X)$ is a polynomial of degree $r$ and the implied constant depends on $r$.
\end{lemma}
\begin{proof}
  The generating Dirichlet series of $\phi(\ell)\ell$ is given by
  \[Z(s)=\prod_{p\nmid q}\left(1-\frac{r}{p^s}\right)^{-1}=\zeta(s)^r\prod_{p\mid q} \left(1-\frac{1}{p^s}\right)^r\prod_{p\nmid q} \left(1-\frac{1}{p^s}\right)^r \left(1-\frac{r}{p^s}\right)^{-1}.\]
Here the last infinite product over $p\nmid q$ converges absolutely in $\Re{s}>1/2$. Hence \eqref{eq:10.7} follows by standard contour integration and the convexity bound $\zeta(s)\ll|s|^{\frac{1}{8}}\log{4|s|}$ on the line $\Re{s}=3/4$. Specifically, the main term $P_r(\log{x})$ in \eqref{eq:10.7} is the residue of $Z(s+1)s^{-1}x^s$ at $s=0$.
\end{proof}

\section{Estimation of $T(X,Y)$}
\label{section:estimationT}

Squaring out \eqref{eq:9.1} we get
\begin{eqnarray*}
  \theta(\ell)^2&=&\mathop{\sum\sum}_{[m_1,m_2]\mid\ell}\mu(m_1)\mu(m_2)g(m_1)g(m_2)h\left(\frac{\ell}{m_1}\right)h\left(\frac{\ell}{m_2}\right)\\
  &=&\mathop{\sum\sum\sum}_{dm_1m_2\mid\ell} \mu(d)\mu(dm_1m_2)g(dm_1)g(dm_2)h\left(\frac{\ell}{dm_1}\right)h\left(\frac{\ell}{dm_2}\right).
\end{eqnarray*}
Note that $dm_1<M$ and $dm_2<M$ by the support of the mollifier. Introducing this into \eqref{eq:9.3} we get
\begin{eqnarray*}
  T(X,Y)=\mathop{\sum\sum\sum\sum}_{X<dm_1m_2\ell\le Y} &&\mu(d)\mu(dm_1m_2)g(dm_1)g(dm_2)\\[-0.3cm]
  &&h(\ell m_1)h(\ell m_2)\phi(dm_1m_2)\phi(\ell).
\end{eqnarray*}
Note that $\phi(dm_1m_2)=\tau_r(d)\tau_r(m_1)\tau_r(m_2)/dm_1m_2$ if $(dm_1m_2,q)=1$ and it vanishes otherwise.

If $d$ is close to $M$, say $M\Delta^{-3}<d<M$, then $m<\Delta^3$ and $g(dm)<(3\log{\Delta}/\log{M})^r$ for $m=m_1m_2$. Hence the contribution of these ``boundary'' terms to $T(X,Y)$ is bounded trivially by
\begin{eqnarray*}
 &&\sum_{M\Delta^{-3}<d<M}\frac{\tau_r(d)}{d}\left(\frac{3\log{\Delta}}{\log{M}}\right)^{2r}\left(\sum_{m<\Delta^3}\frac{\tau_r(m)}{m}\right)^2\sum_{\ell\le Y}\phi(\ell)\\
  &&\ll(\log{M})^{r-1}(\log{\Delta})\left(\frac{\log{\Delta}}{\log{M}}\right)^{2r}(\log{\Delta})^{2r}(\log{Y})^r=\left(\frac{\log{Y}}{\log{M}}\right)^r
\end{eqnarray*}
if we choose $\Delta$ such that $(\log{\Delta})^{4r+1}=\log{M}$. This bound meets the goal \eqref{eq:9.5}. The above lines show how important it is to have the crop function $g(m)$ vanishing at the end point $M$ of degree as large as the degree of the divisor function $\tau_r(d)$.

Now, when $d\le M\Delta^{-3}$ with $(\log\Delta)^{4r+1}=\log{M}$, there is enough room for the Möbius function in $T(X,Y)$ to produce significant cancellation. First, if $m_1>\Delta$ or $m_2>\Delta$, then \eqref{eq:10.4} shows that the contribution of such terms to $T(X,Y)$ is estimated by
\[\exp(-c\sqrt{\log{\Delta}})\left(\sum_{d<M}\tau_r(d)d^{-1}\right)^2 \left(\sum_{\ell<Y}\tau_r(d)\right)\ll \left(\frac{\log{Y}}{\log{M}}\right)^r.\]
This bound meets the goal \eqref{eq:9.5}.

It remains to estimate the partial sum of $T(X,Y)$ over the segment $X<dm_1m_2\ell\le Y$ restricted by the following conditions
\begin{equation}
  \label{eq:11.1}
  m_1\le\Delta, \ m_2\le \Delta, \ d\le M\Delta^{-3}.
\end{equation}
We assume $M\le X<Y\le N$. Then \eqref{eq:11.1} implies $\ell>\Delta$ so we have enough space to execute the summation over $\ell$. By Lemma \ref{lemma:10.1} using partial summation we get
\begin{eqnarray*}
  \sum_{X<dm_1m_2\ell\le Y}\phi(\ell)h(\ell m_1)h(\ell m_2)&=&\int^{Y/d}_{X/d}h\left(\frac{y}{m_1}\right)h\left(\frac{y}{m_2}\right)dP_r\left(\log{\frac{y}{m_1m_2}}\right)\\
  &&+O(\Delta^{-1/5r}).
\end{eqnarray*}
The contribution of the error term to $T(X,Y)$ is $\ll\Delta^{-1/5r}(\log{M})^r(\log{\Delta})^{2r}$ which is much smaller than required. Collecting the above results we get
\begin{equation}
  \label{eq:11.2}
  T(X,Y)=\sum_{d<M\Delta^{-3}}\mu(dq)\frac{\tau_r(d)}{d}\int_{X/d}^{Y/d}T_d(y)\frac{dy}{y}+O\left(\left(\frac{\log{Y}}{\log{M}}\right)^r\right)
\end{equation}
with
\begin{eqnarray*}
  T_d(y)=\mathop{\sum\sum}_{m_1m_2<\Delta} &&\mu(dqm_1m_2)\frac{\tau_r(m_1m_2)}{m_1m_2}g(dm_1)g(dm_2)\\
  &&h\left(\frac{y}{m_1}\right)h\left(\frac{y}{m_2}\right)P'_r\left(\log{\frac{y}{m_1m_2}}\right).
\end{eqnarray*}
Have in mind that the polynomial $P'_r(X)$ has degree $r-1$,
\[P'_r(X)=\sum_{a<r} c(a)X^a.\]
The crop function $g(dm)$ of the mollifier in the above range is the polynomial in $\log{m}/\log{M}$;
\[g(dm)=\left(1-\frac{\log{dm}}{\log{M}}\right)^r=\sum_{0\le j_1\le r}\binom{r}{j}\left(1-\frac{\log{d}}{\log{M}}\right)^{r-j}\left(\frac{-\log{m}}{\log{M}}\right)^j.\]
However
\[h\left(\frac{y}{m}\right)=a \left(\frac{\log{y/m}}{\log{N}}\right)=a \left(\frac{\log{y}}{\log{N}}-\frac{\log{m}}{\log{N}}\right)\]
is not, but it can be approximated by a polynomial using the Taylor expansion
\[h\left(\frac{y}{m}\right)=\sum_{e<E}\frac{1}{e!}a^{(e)}\left(\frac{\log{y}}{\log{N}}\right)\left(\frac{-\log{m}}{\log{N}}\right)^e+O \left(\left(\frac{\log{\Delta}}{\log{N}}\right)^E\right).\]
Choosing $E$ sufficiently large in terms of $r$ the error term becomes negligible. Finally we have
\[P'_r\left(\log{\frac{y}{m_1m_2}}\right)=\sum_{\alpha+\alpha_1+\alpha_2<r} c(\alpha,\alpha_1,\alpha_2)(\log{y})^\alpha(\log{m_1})^{\alpha_1}(\log{m_2})^{\alpha_2},\]
where $c(\alpha,\alpha_1,\alpha_2)=(-1)^{\alpha_1+\alpha_2}c(\alpha+\alpha_1+\alpha_2)(\alpha+\alpha_1+\alpha_2)!/\alpha!\alpha_1!\alpha_2!$. By the above expansions we see that $T_d(y)$ is (up to negligible error terms) a linear combination of sums of type
\begin{eqnarray*}
  \sum_{m_1}\sum_{m_2}&&\mu(dqm_1m_2)\frac{\tau_r(m_1m_2)}{m_1m_2}\left(\frac{\log{m_1}}{\log{M}}\right)^{j_1}\left(\frac{\log{m_2}}{\log{M}}\right)^{j_2}\\
  &&\left(\frac{\log{m_1}}{\log{N}}\right)^{e_1}\left(\frac{\log{m_2}}{\log{N}}\right)^{e_2}(\log{y})^\alpha(\log{m_1})^{\alpha_1}(\log{m_2})^{\alpha_2}
\end{eqnarray*}
with $j_1,j_2\le r$, $e_1,e_2<E$ and $\alpha+\alpha_1+\alpha_2\le r-1$, where the summation is restricted by $m_1\le \Delta$, $m_2\le \Delta$. These restrictions can be dropped up to error term bounded by $(\log{y})^\alpha\exp(-c\sqrt{\log{\Delta}})$, see \eqref{eq:10.5}, which is negligible. The complete sum vanishes, see \eqref{eq:10.6}, unless $j_1+e_1+\alpha_1\ge r$ and $j_2+e_2+\alpha_2\ge r$, in which case it is bounded by
\begin{eqnarray*}
  &&(\log{y})^\alpha(\log{M})^{-j_1-j_2}(\log{N})^{-e_1-e_2}\\
  &\ll&(\log{y})^\alpha(\log{M})^{e_1+e_2+\alpha_1+\alpha_2-2r}(\log{N})^{-e_1-e_2}\\
  &\ll&(\log{y})^\alpha(\log{M})^{\alpha_1+\alpha_2-2r}\\
  &\ll&(\log{y})^\alpha(\log{M})^{-\alpha-r-1}\\
  &\ll&(\log{y})^{r-1}(\log{M})^{-2r}.
\end{eqnarray*}
Inserting this bound into \eqref{eq:11.2} we derive \eqref{eq:9.5}.

\section{Conclusion}
\label{section:conclusion}
We have all parts (except for the off-diagonal terms) ready to conclude the proof of the main Theorem \ref{thm:2.1}. Take the mollifier \eqref{eq:3.13} of length $M=T^{1/400}\ge Q^8$. Choose the breaking points in the partition \eqref{eq:4.1} at $\alpha=\frac{1}{2}+\frac{1}{100}$, $\beta=\frac{1}{2}+\frac{1}{200}$. Then the level of Levinson's function \eqref{eq:3.2} satisfies $T^2<N<T^3$, see \eqref{eq:4.14}. The diagonal sum \eqref{eq:6.4} is estimated by
\[\Kc_0=S(1,MN^\alpha)\le S(1,N)=S(1,Q^8)+S(Q^8,M)+S(M,N).\]
Applying \eqref{eq:7.3}, \eqref{eq:7.4} and \eqref{eq:9.7} we get
\[\Kc_0\ll \left(\frac{\log{Q}}{\log{T}}\right)^2+L(1,\chi)\log{T}+(L(1,\chi)\log{T})^{\frac{1}{2}}.\]
The same bound holds for the diagonal sum $\Kc_0^*$ given by \eqref{eq:6.2}. Hence \eqref{eq:5.7} yields
\[I(T)\le T|\Kc^\neq|^{\frac{1}{2}}+O\left(T \frac{\log{Q}}{\log{T}}+T(L(1,\chi)\log{T})^{\frac{1}{2}}+T(L(1,\chi)\log{T})^{\frac{1}{4}}\right)\]
where the implied constant is absolute. Inserting this into \eqref{eq:3.10} we get
\begin{eqnarray}
  \label{eq:12.1}
  N_{00}(T)&>&N(T)-4|\Kc^\neq|^{\frac{1}{2}}T\log{T}\\
  &&+O\left(T\log{Q}+T(\log{T})(L(1,\chi)\log{T})^{\frac{1}{4}}\right).\nonumber
\end{eqnarray}
Note that the condition $T\ge Q^{3200}$ is no longer required, because the estimate \eqref{eq:12.1} holds trivially otherwise. It remains to estimate the contribution $\Kc^\neq(T)$ of the off-diagonal terms, see \eqref{eq:6.5}.

\section{An Introduction to the Off-diagonal Terms}
\label{section:introductionOffDiagonalTerms}
Our goal is to show that the contribution of $\Kc^\neq(T)$ to $\Kc(T)$ is quite small, comparable to $\Kc_0$, so that it can be omitted in \eqref{eq:12.1}. We shall only consider the $L$-function for the trivial ideal class group character $\psi=\psi_0$, in which case
\begin{equation}
  \label{eq:13.1}
  L(s)=\zeta(s)L(s,\chi)=\sum_n \lambda(n)n^{-s}
\end{equation}
with $\lambda=1\star\chi$. The other cases are similar, in fact simpler, because the main term of $\Kc^\neq(T)$ vanishes.

Before starting advanced arguments we recall the situation in fresh notation to recycle a lot of alphabet which was used so far. We have
\begin{equation}
  \label{eq:13.2}
  \Kc^\neq(T)=\mathop{\sum\sum}_{u,v<M} \frac{\rho(u)\rho(v)}{\sqrt{uv}}g(u)g(v)I\left(\frac{u}{v}\right)+O\left(\frac{1}{T}\right)
\end{equation}
where
\begin{equation}
  \label{eq:13.3}
  I\left(\frac{u}{v}\right)=\mathop{\sum\sum}_{um\neq vn}\Psi\left(T\log{\frac{um}{vn}}\right)\frac{\lambda(m)\lambda(n)}{\sqrt{mn}}h(m)h(n)
\end{equation}
and $\Psi(2\pi z)$ denotes the Fourier transform of $\Phi(t)$. The error term $O(1/T)$ in \eqref{eq:13.2} is an easy estimate for the contribution of terms $um=1$ or $vn=1$ which are added in \eqref{eq:13.3}.

In \eqref{eq:5.5} we said that $\Phi(t)$ was smooth and compactly supported on $\R^+$. Clearly we can modify it here by requesting the symmetry $\Phi(t)=\Phi(-t)$. Then we have $\Psi(z)=\Psi(-z)$ and $I(u/v)=I(v/u)$. Note that
\begin{equation}
  \label{eq:13.4}
  \Psi(z)=\int_{-\infty}^\infty\Phi(t)\cos(tz)dt
\end{equation}
has fast decaying derivatives, specifically we shall often use the bound
\begin{equation}
  \label{eq:13.5}
  \Psi^{(j)}(z)\ll(1+|z|)^{-A}, \ j=0,1,2,
\end{equation}
for real $z$ with any $A\ge 4$ where the implied constant depends on $A$.

Pulling out the greatest common factor of $u,v$ in \eqref{eq:13.2} we write
\begin{equation}
  \label{eq:13.6}
  \Kc^\neq(T)=\sum_{e< M}\mathop{\sum\sum}_{\substack{u,v<M/e\\ (u,v=1)}} \frac{\rho(eu)\rho(ev)}{e\sqrt{uv}}g(eu)g(ev)I\left(\frac{u}{v}\right)+O\left(\frac{1}{T}\right)
\end{equation}
with $I(u/v)$ given by \eqref{eq:13.3} without change. Given $(u,v)=1$ we split $I(u/v)$ into
\begin{equation}
  \label{eq:13.7}
  I\left(\frac{u}{v}\right)=2\sum_{h=1}^\infty I_h\left(\frac{u}{v}\right)
\end{equation}
where
\begin{equation}
  \label{eq:13.8}
  I_h\left(\frac{u}{v}\right)=\mathop{\sum\sum}_{um-vn=h} \Psi\left(T\log{\frac{um}{vn}}\right)\frac{\lambda(m)\lambda(n)}{\sqrt{mn}}h(m)h(n)
\end{equation}
are additive convolution type sums.

\begin{remarks*}
  Since $\Phi(z)$ decays rapidly $T\log{\frac{um}{vn}}=T\log{\left(1+\frac{h}{vn}\right)}$ is essentially bounded so $vn\gg hT\ge T$, $um\gg hT\ge T$ and $um$, $vn$ are close to each other, $um/vn=1+O(1/T)$.
\end{remarks*}

The notation begins to be cumbersome so in the next three sections we are going to present self-contained results about additive convolution sums which will be applicable to \eqref{eq:13.8}.

\section{General Convolution Sums}
\label{section:generalConvolutionSums}

This is a stand-alone section. Here and in the next three sections our notation is independent of that used in the previous ones. After proving Lemma \ref{lemma:17.1} we shall abandon this temporary notation.

Suppose we are given two sequences $\Ac=(a_m)$, $\Ac^*=(a_n^*)$, which enjoy some features of the Fourier coefficients of automorphic forms. Our goal is to evaluate the sum
\begin{equation}
  \label{eq:14.1}
  \Bc(h)=\sum_{m-n=h}a_ma_n^* g(m)g^*(n)
\end{equation}
for $h\ge 1$, where $g(x),g^*(x)$ are smooth functions, compactly supported on $\R^+$. Sums of such type were treated in Section 4 of \cite{CI} in a great generality using ideas of Kloosterman's circle method. Now we need \eqref{eq:14.1} in a little bit more general setting, in which case the arguments in Section 4 of \cite{CI} still apply. Since the required modifications are essentially in the notation we shall state the results without repeating proofs. If the sequences $\Ac,\Ac^*$ consist of Hecke eigenvalues for a cusp form there are several results in the literature which are useful for out applications. In particular the formulas of \cite{KMV} in Appendices $A,B$ come close to what we require with respect to the parameters out of which to built the mollifier. However the shifted convolution for $\lambda=1\star\chi$ is not covered in \cite{KMV}. This paper \cite{KMV} contains numerous fundamental ideals and gives great details so we recommend to the reader to glance it as a supplement to our arguments below.

All we need about the sequence $\Ac$ (and $\Ac^*$) is a kind of Voronoi formula for twisted sums
\begin{equation}
  \label{eq:14.2}
  S(\alpha)=\sum_m a_m g(m)e(\alpha m)
\end{equation}
at rational points $\alpha=a/c$ for every $c\ge 1$ and $(a,c)=1$. Naturally, one expects that $S(a/c)$ are quite well approximated by
\begin{equation}
  \label{eq:14.3}
  \psi(a,c)\int g(x)dx
\end{equation}
where $\psi(a,c)$ is a nice function which depends on $a \pmod{c}$ and it satisfies
\begin{equation}
  \label{eq:14.4}
  |\psi(a,c)|\le \frac{A}{c}.
\end{equation}
Here the parameter $A$, and two other parameters $B,C$ in forthcoming conditions, will be specified in later applications subject to $A\ge 1, \ B\ge 1, \ C\ge 2$. Note that $\psi(a,c)$ does not depend on the test function $g(x)$, therefore the approximation \eqref{eq:14.3} to the sum $S(a/c)$ is a functional.

Based on \eqref{eq:14.3} one should predict that $\Bc(h)$ is quite well approximated by
\begin{equation}
  \label{eq:14.5}
  B(h)=\Sf(h)\int g(x+h)g^*(x)dx
\end{equation}
where
\begin{equation}
  \label{eq:14.6}
  \Sf(h)=\sum_{c=1}^\infty \hspace{0.2cm} \sideset{}{^*}\sum_{a\pmod{c}}e\left(\frac{ah}{c}\right)\psi(-a,c)\psi^*(a,c).
\end{equation}
Indeed we shall see that under suitable conditions the prediction is pretty accurate. We assume that the Fourier transform of $g(x)$ satisfies
\begin{equation}
  \label{eq:14.7}
  \int |\hat g(\alpha)|d\alpha\le B, \hspace{0.5cm} \int |\alpha||\hat g(\alpha)|^2 d\alpha\le B^2.
\end{equation}
Moreover the same estimates hold for the Fourier transform of $g^*(x)$. We write
\begin{equation}
  \label{eq:14.8}
  S\left(\frac{a}{c}\right)=\psi(a,c)\int g(x)dx+T(a,c),
\end{equation}
where the error term $T(a,c)$ does, of course, depend on $g(x)$ as a functional. It is not sufficient to assume a good upper bound for $T(a,c)$; one has to control the variation of its argument and get a  considerable cancellation when summing over the classes $a\pmod{c}$, $(a,c)=1$ (this is the very essence of Kloosterman's circle method). We postulate that every $T(a,c)$ has the Fourier series expansion of the following type
\begin{equation}
  \label{eq:14.9}
  T(a,c)=\sum_{m=1}^\infty \psi_m(a)e\left(\frac{\overline a}{c}\ell_m\right)\int g(x)k_m(x)dx
\end{equation}
where $\overline a$ denotes the multiplicative inverse of $a$ modulo $c$, $a\overline a\equiv 1\pmod{c}$. Here the frequencies $\ell_m$ are integers which are allowed to depend on $c$, but not on $a$. Moreover the kernel functions $k_m(x)$ may depend on $c$, but not on $a$. Finally, the coefficients $\psi_m(a)$ are also allowed to depend on $c$ in an arbitrary fashion, but the dependence on $a$ must be mild. Specifically, we assume that there is a fixed integer $q\ge 1$ such that $\psi_m(a)$ is periodic in $a$ modulo $(c,q)$, and
\begin{equation}
  \label{eq:14.10}
  |\psi_m(a)|\le \frac{A}{c}\tau(m).
\end{equation}
We also assume that the Fourier transform of $g_m(x)=g(x)k_m(x)$ satisfies
\begin{equation}
  \label{eq:14.11}
  |\hat g_m(\alpha)|\le cCBm^{-5/4}
\end{equation}
for every $1\le c\le C$ and every $\alpha$ with $|\alpha|cC\le 1$.

Now we are ready to state the following result (go through Section 4 of \cite{CI} line by line for constructing a definite proof).

\begin{proposition}\label{prop:14.1}
  Assume the conditions \eqref{eq:14.4}, \eqref{eq:14.7}, \eqref{eq:14.10}, \eqref{eq:14.11} for the sequence $\Ac=(a_m)$ and the corresponding conditions for the sequence $\Ac^*=(a_n^*)$. Then for every $h\ge 1$ we have
  \begin{equation}
    \label{eq:14.12}
    \Bc(h)=B(h)+R(h)
  \end{equation}
  where $\Bc(h)$ is the convolution sum \eqref{eq:14.1}, $B(h)$ is the predicted main term \eqref{eq:14.5} and $R(h)$ is an error term which satisfies
  \begin{equation}
    \label{eq:14.13}
    R(h)\ll \tau(h)A^2C^{-1}\int |g(x+h)g^*(x)|dx+\tau(h)qA^2B^2C^{\frac{3}{2}}(\log{C})^2
  \end{equation}
  with the implied constant being absolute.
\end{proposition}

\section{Special Convolution Sums}
\label{section:specialConvolutionSums}
We are interested in the sequence $\lambda=1\star\chi$ as in \eqref{eq:1.8} where $\chi$ is the real primitive character of conductor $|D|$. By Proposition 3.3 of \cite{CI} we have the following Voronoi type formula
\begin{equation}
  \label{eq:15.1}
  \sum_{m=1}^\infty \lambda(m)e\left(\frac{a}{c}m\right)g(m)=\rho(a,c)L(1,\chi)\int g(x)dx+T(a,c)
\end{equation}
for any $c\ge 1$, $(a,c)=1$, where
\begin{eqnarray*}
  T(a,c)=2\pi i\chi_1(a)\chi_2(c) \frac{\sqrt{(c,D)}}{c\sqrt{|D|}} &&\sum_{m=1}^\infty (\chi_1\star\chi_2)(m)e\left(\overline{aD/(c,D)}\frac{m}{c}\right)\\
  &&\int g(x)J_0\left(4\pi\sqrt{(c,D)mx}/c\sqrt{|D|}\right)dx,
\end{eqnarray*}
$\chi_1\pmod{(c,D)}$ and $\chi_2\pmod{|D|/(c,D)}$ are the real characters such that $\chi_1\chi_2=\chi$. In the main term we have
\begin{equation}
  \label{eq:15.3}
  \rho(a,c)=
  \begin{cases}
    \chi(c)/c&\text{ if }D\nmid c,\\
    \chi(a)\tau(\chi)/c&\text{ if }D\mid c
  \end{cases}
\end{equation}
where $\tau(\chi)$ denotes the Gauss sum. Actually Proposition 3.3 of \cite{CI} requires $D$ to be odd and negative, so
\begin{equation}
  \label{eq:15.4}
  D\text{ is squarefree }, \ D<0, \ D\equiv 1\hspace{-0.2cm}\pmod{4}.
\end{equation}
Therefore, in the next three sections we shall be working under these conditions. The other cases are very much similar and the final estimates are the same so we skip them.

For every positive integer $u$ we derive from \eqref{eq:15.1} the following formula
\begin{eqnarray*}
  \sum_{m=1}^\infty \lambda(m)e\left(\frac{a}{c}um\right)g(um)&=&\rho \left(\frac{au}{(c,u)},\frac{c}{(c,u)}\right)\frac{L(1,\chi)}{u}\int g(x)dx\\
  &&+ T(au/(c,u),c/(c,u)).
\end{eqnarray*}
(replace $a,c,g(x)$ in \eqref{eq:15.1} by $au/(c,u), \ c/(c,u), \ g(ux)$ respectively).

Now we can apply Proposition \ref{prop:14.1} for the sequences $\Ac=(a_m)$, $\Ac^*=(a_n^*)$ with $a_m=\lambda(m/u)$, $a_n^*=\lambda(n/v)$, where $u,v$ are given positive integers (subject to the popular convention that an arithmetic function is set its value to zero at non-integers arguments). Suppose $g(x)$ and $g^*(x)$ are smooth functions supported in a dyadic segment $[X,2X]$ with $X\ge 2$ whose derivatives satisfy
\begin{equation}
  \label{eq:15.5}
  |x^jg^{(j)}(x)|\le 1, \ j=0,1,2.
\end{equation}
Then one can show (see the arguments in Section 4 of \cite{CI}) that \eqref{eq:14.7} holds with $B\ll 1$ and \eqref{eq:14.11} holds for $C=2(u+v)\sqrt{|D|X}$ with $B\ll (u+v)|D|^{\frac{3}{2}}$. Moreover \eqref{eq:14.4} holds with $A\ll \sqrt{|D|}L(1,\chi)\ll \sqrt{|D|}\log{|D|}$ and \eqref{eq:14.11} holds for $q=|D|$ with $A\ll u+v$. Therefore Proposition \ref{prop:14.1} yields
\begin{proposition}\label{prop:15.1}
  Let $g(x),g^*(x)$ be smooth functions supported in $[X,2X]$ with $X\ge 2$ whose derivatives satisfy \eqref{eq:15.5}. Then for positive integers $u,v,h$ we have
  \begin{eqnarray}
    \label{eq:15.6}
    &&\hspace{2.6cm}\sum_{um-vn=h} \lambda(m)\lambda(n)g(um)g^*(vn)=\\
    &&\Sf(h)(uv)^{-1}L^2(1,\chi)\int g(x+h)g^*(x)dx+O\left(\tau(h)(uvD)^6X^{\frac{3}{4}}(\log{X})^2\right) \nonumber
  \end{eqnarray}
where
\begin{equation}
  \label{eq:15.7}
  \Sf(h)=\sum_{c=1}^\infty\hspace{0.2cm}\sideset{}{^*}\sum_{a\pmod{c}} e\left(\frac{ah}{c}\right)\rho \left(\frac{-au}{(c,u)},\frac{c}{(c,u)}\right)\rho \left(\frac{av}{(c,v)},\frac{c}{(c,v)}\right)
\end{equation}
with $\rho(a,c)$ given by \eqref{eq:15.3} and the implied constant being absolute.
\end{proposition}

\begin{remarks*}
  We have not assumed that $u,v$ are co-prime. But, of course, if $(u,v)\nmid h$, then the convolution sum on the left side of \eqref{eq:15.6} is void so the series $\Sf(h)$ on the right side of \eqref{eq:15.6} must vanish as well. This could be verified directly if you will, but not so easily. The result is a generalization of a special case of Theorem 4.4 of \cite{CI}. Note that the exponent $3/4$ in the error term comes from an application of Weil's bound for Kloosterman sums. The exponent $7/8$ resulting from a weaker elementary bound due to Kloosterman would be also sufficient for our purpose.
\end{remarks*}

For technical simplifications we can impose some local restrictions on the variables $u,v$ in the formula \eqref{eq:13.2}. These numbers will be in the support of the coefficients $\rho(u),\rho(v)$ of the mollifier, see \eqref{eq:3.11} and \eqref{eq:3.12}. Therefore we can assume that $u$ and $v$ are cubefree with no multiple ramified prime factors; this means $p\mid D\Rightarrow p^2\nmid u\text{ and }p^2\nmid v$.
\section{Computing the Series $\Sf(h)$}
\label{section:ComputingSf}

To ease the computations we assume that $u,v$ are coprime;
\begin{equation}
  \label{eq:16.1}
  (u,v)=1,
\end{equation}
so
\begin{equation}
  \label{eq:15.8}
  (D^2,uv)=(D,uv).
\end{equation}
Let $1(x)$ denote the characteristic function of integers. Then \eqref{eq:15.3} yields $\rho(a,c)c=\chi(c)+\chi(a)\tau(\chi)1(c/D)$ and \eqref{eq:15.7} becomes
\[\Sf(h)=\sum_{c=1}^\infty \frac{(c,uv)}{c^2}\sideset{}{^*}\sum_{a\pmod{c}} e\left(\frac{ah}{c}\right)\{\dots\}\{\dots\}\]
where
\begin{eqnarray*}
  \{\dots\}\{\dots\}&=&\left\{\chi\left(\frac{c}{(c,u)}\right)-\chi\left(\frac{-au}{(c,u)}\right)\tau(\chi)1\left(\frac{c}{(c,u)D}\right)\right\}\\
  &&\left\{\chi\left(\frac{c}{(c,v)}\right)+\chi\left(\frac{av}{(c,v)}\right)\tau(\chi)1\left(\frac{c}{(c,v)D}\right)\right\}\\
  &=&\chi\left(\frac{c^2}{(c,uv)}\right)-\chi\left(\frac{a^2uv}{(c,uv)}\right)1\left(\frac{c}{(c,uv)D}\right)D\\
  &&+\left[\chi\left(\frac{acv}{(c,uv)}\right)1\left(\frac{c}{(c,v)D}\right)-\chi\left(\frac{acu}{(c,uv)}\right)1\left(\frac{c}{(c,u)D}\right)\right]\tau(\chi)
\end{eqnarray*}
because $\chi(-1))=-1$, $\tau(\chi)^2=D$ and $1(c/(c,u)D)1(c/(c,v)D)=1(c/(c,uv)D)$. Note that except for the first term $\chi(c^2/(c,uv))$ the other three terms vanish unless $D\mid c$ in which case $\chi(a^2)=1$. By this observation we get
\begin{eqnarray*}
  \{\dots\}\{\dots\}=&&\chi\left(\frac{c^2}{(c,uv)}\right)-\chi\left(\frac{uv}{(c,uv)}\right)1\left(\frac{c}{(c,uv)D}\right)D\\
  &&+(\chi(v)-\chi(u))\chi\left(\frac{ac}{(c,uv)}\right)1\left(\frac{c}{D}\right)\tau(\chi).
\end{eqnarray*}
Next we introduce the Ramanujan sum
\begin{equation}
  \label{eq:16.2}
  r_h(c)=\sideset{}{^*}\sum_{a\pmod{c}} e\left(\frac{ah}{c}\right)=\sum_{d\mid (c,h)} d\mu(c/d)
\end{equation}
and if $D\mid c$ we introduce the hybrid of Gauss-Ramanujan sum
\begin{equation}
  \label{eq:16.3}
  r_h(c,\chi)=\sideset{}{^*}\sum_{a\pmod{c}} \chi(a)e\left(\frac{ah}{c}\right).
\end{equation}
Put
\begin{equation}
  \label{eq:16.4}
  uv=w,\hspace{0.5cm}\text{ so }w\text{ is cubefree.}
\end{equation}
The symbols $1(c/(c,w)d)$ and $1(c/D)$ above mean that $(D,w)D\mid c$ (see the condition \eqref{eq:15.8}) and $D\mid c$, respectively. Therefore $\Sf(h)$ splits into three parts
\begin{equation}
  \label{eq:16.5}
  \Sf(h)=\Sf^*(h)-\Sf'(h)+(\chi(v)-\chi(u))\Sf(h,\chi)
\end{equation}
where
\begin{equation}
  \label{eq:16.6}
  \Sf^*(h)=\sum_{(c,D)=1} \chi((c,w))r_h(c)(c,w)c^{-2},
\end{equation}
\begin{equation}
  \label{eq:16.7}
  \Sf'(h)=D\sum_{(D,w)D\mid c}\chi(w/(c,w))r_h(c)(c,w)c^{-2},
\end{equation}
\begin{equation}
  \label{eq:16.8}
  \Sf(h,\chi)=\tau(\chi)\sum_{D\mid c}\chi(c/(c,w))r_h(c,\chi)(c,w)c^{-2}.
\end{equation}

Note that $\Sf(h,\chi)$ vanishes, unless $D\mid w$, and the third part of \eqref{eq:16.5} vanishes, unless $D\mid u$ or $D\mid v$. These are pretty strong conditions on $u,v$ which we can easily go around in applications. Therefore, from now on we assume that
\begin{equation}
  \label{eq:16.9}
  D\nmid u\text{ and }D\nmid v
\end{equation}
so the third part of \eqref{eq:16.5} does not need to be considered
(see Section~\ref{section:commentsCompletingProof}).

By the formula \eqref{eq:16.2} we can write the first and the second parts of \eqref{eq:16.5} as the convolutions $1\star \gamma^*$ and $1\star\gamma'$, say, with
\begin{equation}
  \label{eq:16.10}
  \gamma^*(d)=\frac{1}{d}\sum_{(cd,D)=1}\chi((cd,w))(cd,w)\mu(c)c^{-2}
\end{equation}
and
\begin{equation}
  \label{eq:16.11}
  \gamma'(d)=\frac{D}{d}\sum_{(D,w)D\mid cd}\chi(w/(cd,w))(cd,w)\mu(c)c^{-2}.
\end{equation}
Finally, assuming the conditions \eqref{eq:16.9} we conclude that
\begin{equation}
  \label{eq:16.12}
  \Sf(h)=\Sf^*(h)-\Sf'(h)=(1\star\gamma^*)(h)-(1\star\gamma')(h).
\end{equation}

\begin{lemma}\label{lemma:16.1}
  We have $\gamma^*(d)=0$ unless $(d,D)=1$ in which case
  \begin{equation}
    \label{eq:16.13}
    \gamma^*(d)=\frac{(d,w)}{\zeta(2)d}\chi((d,w))\xi(w/(d,w))
  \end{equation}
  where
  \begin{equation}
    \label{eq:16.14}
    \xi(n)=\prod_{p\mid n}\left(1+\frac{\chi(p)}{p}\right)^{-1}.
  \end{equation}
\end{lemma}
\begin{proof}
  First note that the product \eqref{eq:16.14} for $n=w/(d,w)$ runs over the set $\Pc$ of primes $p\mid w$ such that
  \begin{equation}
    \label{eq:16.15}
    p\parallel w\Rightarrow p\nmid d\text{ and }p^2\mid w\Rightarrow p^2\nmid d.
  \end{equation}
  Clearly the sum \eqref{eq:16.10} is void if $(d,D)\neq 1$. If $(d,D)=1$, then $\gamma^*(d)=\chi((d,w))(d,w)d^{-1}\Sigma$, where
  \begin{eqnarray}
    \label{eq:16.16}
    \Sigma&=&\sum_{(c,D)=1} \chi\left(\frac{(cd,w)}{(d,w)}\right)\frac{(cd,w)}{(d,w)}\frac{\mu(c)}{c^2}\\
    &=&\prod_{p\in\Pc} \left(1-\frac{\chi(p)}{p}\right)\prod_{p\not\in\Pc} \left(1-\frac{1}{p^2}\right)\nonumber\\
    &=&\zeta(2)^{-1}\prod_{p\in\Pc}\left(1-\frac{\chi(p)}{p}\right)\left(1-\frac{1}{p^2}\right)^{-1}.\nonumber
  \end{eqnarray}
This yields the formula \eqref{eq:16.13}.
\end{proof}

Note that \eqref{eq:16.13} gives the upper bound (not to be used)
\begin{equation}
  \label{eq:16.17}
  |\gamma^*(d)|\le \frac{(d,w)}{d}\prod_{p\mid w}\left(1+\frac{1}{p}\right).
\end{equation}

\begin{lemma}\label{lemma:16.2}
  We have $\gamma'(d)=0$, unless $(D,w)\mid d$ in which case
  \begin{equation}
    \label{eq:16.18}
    \gamma'(d)=\frac{\mu(q)}{\zeta_q(2)}\frac{(d_1,w_1)(d_1,D)^2}{dD}\chi\xi \left(\frac{w_1}{(d_1,w_1)}\right)
  \end{equation}
  where $w_1=w/(D,w), \ d_1=d/(D,w), \ q=D/(D,d_1)$, $\zeta_q(s)$ denotes the Riemann zeta function with missing local factors at $p\mid q$ and $\xi(n)$ is given by \eqref{eq:16.14}.
\end{lemma}
\begin{proof}
  By \eqref{eq:15.8} it follows that $(D,w_1)=1$. Clearly the sum \eqref{eq:16.11} is void if $(D,w)$ does not divide $d$, because $c$ is squarefree. If $d=(D,w)d_1$, then the formula becomes \eqref{eq:16.11}
  \[\gamma'(d)=\frac{D}{d}\sum_{D\mid cd_1} \chi(w_1/(cd_1,w_1))(cd_1,w_1)\mu(c)c^{-2}.\]
Here $c=c_1D/(D,d_1)$ and $(cd_1,w_1)=(cd_1,w_1)=(d_1,w_1)(c_1,w_1/(d_1,w_1))$ giving
\[\gamma'(d)=\mu \left(\frac{D}{(D,d_1)}\right)\frac{(D,d_1)^2}{dD}(d_1,w_1)\chi \left(\frac{w_1}{(d_1,w_1)}\right)\Sigma\]
where
\begin{eqnarray*}
  \Sigma&=&\sum_{(c_1,D/(D,d_1))=1}\chi \left(\left(c_1,\frac{w_1}{(d_1,w_1)}\right)\right)\frac{\mu(c_1)}{c_1^2}\\
  &=&\prod_{p\nmid \frac{D}{(D,d_1)}\frac{w_1}{(d_1,w_1)}} \left(1-\frac{1}{p^2}\right)\prod_{p\mid \frac{w_1}{(d_1,w_1)}} \left(1-\frac{\chi(p)}{p}\right)\\
  &=&\prod_{p\nmid \frac{D}{(D,d_1)}} \left(1-\frac{1}{p^2}\right)\prod_{p\mid \frac{w_1}{(d_1,w_1)}} \left(1+\frac{\chi(p)}{p}\right)^{-1}.
\end{eqnarray*}
Hence it is easy to check the formula \eqref{eq:16.18}.
\end{proof}

Note that \eqref{eq:16.18} gives the upper bound (not to be used)
\begin{equation}
  \label{eq:16.19}
  |\gamma'(d)|\le \frac{(d_1,w_1)}{(D,w)} \frac{(D,d_1)^2}{|D| d_1}\prod_{p\mid w_1} \left(1+\frac{1}{p}\right).
\end{equation}

\begin{remarks*}
  The formula \eqref{eq:16.18} for $\gamma'(d)$ is very similar to the formula \eqref{eq:16.13} for $\gamma^*(d)$, but it is a bit more involved with respect to the ramified prime places. However, both expressions share the same essential features which are relevant to the forthcoming analysis of the series $\Sf^*(h)$, $\Sf'(h)$ and $k(y)=k^*(y)-k'(y)$, see \eqref{eq:17.6}. Therefore, we are going to work with $\gamma^*(d),\Sf^*(h),k^*(y)$ in considerable details and skip the analysis of $\gamma'(d)$, $\Sf'(h)$, $k'(y)$ assuming that the final estimates are the same in both cases. Actually, the case of $\gamma'(d)$ in some extreme situations can be treated somewhat faster, because a crude upper bound for $\gamma'(d)$ is good enough. For example, suppose $d$ has no ramified prime factors, i.e.
  \begin{equation}
    \label{eq:16.23}
    (d,D)=1.
  \end{equation}
  Then $\gamma'(d)=0$, unless $(D,w)=1$, in which case $d=d_1$, $w=w_1$, $D'=1$ and \eqref{eq:16.18} becomes
  \begin{equation}
    \label{eq:16.20}
    \gamma'(d)=\frac{\mu(D)}{\zeta_D(2)}\frac{(d,w)}{dD}\chi \left(\frac{w}{(d,w)}\right)\xi \left(\frac{w}{(d,w)}\right).
  \end{equation}
Here we gained the factor $1/|D|$ by comparison to \eqref{eq:16.13} which is so small that one can cover the range $(d,D)=1$ exploiting neither the lacunarity of $L(s)=\zeta(s)L(s,\chi)$ nor the sifting effects of the action of the mollifier $M(s)$.

In other extreme example suppose that every ramified prime divides $d$ but not $w$, i.e.
\begin{equation}
  \label{eq:16.21}
  D\mid d, \ (D,w)=1.
\end{equation}
Then \eqref{eq:16.18} reduces to \eqref{eq:16.13}, precisely we have
\begin{equation}
  \label{eq:16.22}
  \gamma'(d)=\chi(w)\gamma^*(d/|D|).
\end{equation}
\end{remarks*}

\section{Back to the Off-Diagonal Terms}
\label{section:backOffDiagonalTerms}

We are now ready to evaluate the components $I(u/v)$ of the sum $\Kc^\neq(T)$, see \eqref{eq:13.2} and \eqref{eq:13.1}. First by Proposition \ref{prop:15.1} we derive
\begin{lemma}\label{lemma:17.1}
  Let $h,u,v$ be positive integers with $(u,v)=1$, $u,v<M$. Then
  \begin{equation}
    \label{eq:17.1}
    I_h\left(\frac{u}{v}\right)=\frac{\Sf(h)}{\sqrt{uv}}L^2(1,\chi)\int_0^\infty \Psi\left(\frac{hT}{x}\right)h\left(\frac{x}{u}\right)h\left(\frac{x}{v}\right)\frac{dx}{x}+O\left(h^2T^{-\frac{1}{9}}\right).
  \end{equation}
\end{lemma}
\begin{proof}
  The left side of \eqref{eq:15.6} with
  \[g(x)=h\left(\frac{x}{u}\right)\left(\frac{u}{x}\right)^{\frac{1}{2}}, \ g^*(y)=\Psi\left(T\log{\frac{y+h}{y}}\right)h\left(\frac{y}{v}\right)\left(\frac{v}{y}\right)^{\frac{1}{2}}\]
becomes $I_h(u/v)$ and the main term on the right side of \eqref{eq:15.6} becomes
\begin{equation}
  \label{eq:17.2}
  \frac{\Sf(h)}{\sqrt{uv}}L^2(1,\chi)\int_0^\infty \Psi\left(T\log{\frac{x+h}{x}}\right)h\left(\frac{x+h}{u}\right)h\left(\frac{x}{v}\right)\frac{dx}{\sqrt{x(x+h)}}.
\end{equation}
However, the conditions of Proposition \ref{prop:15.1} are not exactly satisfied by the above choice of the test functions $g(x),g^*(y)$. To meet the conditions \eqref{eq:15.5} we apply a smooth partition of the summation variables with constituents supported in semi-dyadic segments $[X,\sqrt{2}X]$, $[Y,\sqrt{2}Y]$ so that our partial sums run over $m,n$ with $X<um<\sqrt{2}X$, $Y<vn<\sqrt{2}Y$ for some $1/2\le X,Y\le MN^\alpha$. If the segments are equal or are adjacent to each other, then we are dealing with two test functions supported in the same dyadic segment. Moreover the derivatives satisfy $x^jg^{(j)}(x)\ll (u/X)^{\frac{1}{2}}$ and $y^j{g^*}^{(j)}(y)\ll (v/Y)^{\frac{1}{2}}$. Therefore \eqref{eq:15.6} yields the main term \eqref{eq:17.2} for such partial sums with an error term bounded by
\begin{equation}
  \label{eq:17.3}
  \tau(h)(uvD)^6(uv)^{\frac{1}{2}}X^{-\frac{1}{4}}(\log{4X})^2.
\end{equation}
We have chosen $\alpha$ slightly larger than $1/2$, $\alpha=51/100$, and $M$ a relatively small power of $T$, $M=T^{1/400}$. Hence our segments end at $2X,2Y<2MN^\alpha=2M(QT)^{2\alpha}<|D|MT^{2\alpha}<T^{45/44}$, provided $|D|<T^{1/4000}$. We have already said in the Remarks following \eqref{eq:13.8} that $\Psi(z)$ decays rapidly, see \eqref{eq:13.5}. Therefore, the contribution to $I_h(u/v)$ of the partial sums over the segments $[X,\sqrt{2}X], \ [Y,\sqrt{2}Y]$ is negligible, expect for
\[hT^{43/44}<X\le \sqrt{2}Y\le 2X\le T^{45/44}\]
in which cases \eqref{eq:17.3} is much smaller than the error term in \eqref{eq:17.1}. Furthermore, since $h$ is quite small, $h<T^{1/22}$, we can clear the main term \eqref{eq:17.2} by means of the following approximations;
\begin{eqnarray*}
  \Psi\left(T\log{\frac{x+h}{x}}\right)&=&\Psi\left(\frac{hT}{x}\right)+O(Th^2x^{-2}),\\
  h\left(\frac{x+h}{u}\right)&=&h\left(\frac{x}{u}\right)+O(h/x),\\
  (x+h)^{-\frac{1}{2}}&=&x^{-\frac{1}{2}}(1+O(h/x)).
\end{eqnarray*}
The error terms in the above approximations are negligible. This completes the proof of Lemma \ref{lemma:17.1}.
\end{proof}

\begin{remarks*}
  Having derived the formula \eqref{eq:17.1} from the results in Sections \ref{section:generalConvolutionSums}, \ref{section:specialConvolutionSums} we no longer need these sections, in particular the test functions $g,g^*$ used over there can be forgotten. From now $g$ stands again for the crop function in the mollifier \eqref{eq:3.13}.
\end{remarks*}

Introducing \eqref{eq:17.1} into \eqref{eq:13.7} we obtain
\begin{equation}
  \label{eq:17.4}
  I\left(\frac{u}{v}\right)=\frac{L^2(1,\chi)}{\sqrt{uv}}\int_0^\infty k\left(\frac{T}{x}\right)h\left(\frac{x}{u}\right)h\left(\frac{x}{v}\right)\frac{dx}{x}+O(T^{-\frac{1}{9}})
\end{equation}
where
\begin{equation}
  \label{eq:17.5}
  k(y)=2\sum_{h>0} \Sf(h)\Psi(hy).
\end{equation}
Then, introducing $\Sf=1\star\gamma$ with $\gamma(d)=\gamma^*(d)-\gamma'(d)$, see \eqref{eq:16.12}, we get
\begin{equation}
  \label{eq:17.6}
  k(y)=\sum_d \gamma(d)\phi(dy)
\end{equation}
where
\begin{equation}
  \label{eq:17.7}
  \phi(z)=\sum_k \Psi(kz).
\end{equation}
Do not forget that $\gamma(d)$ depends on $w=uv$, see \eqref{eq:16.13} and \eqref{eq:16.18}.

Finally we insert \eqref{eq:17.4} into \eqref{eq:13.6} to get a complete formula for the off-diagonal contribution

\begin{lemma}\label{lemma:17.2}
  We have
  \begin{eqnarray}
    \label{eq:17.8}
    \Kc^\neq(T)&=&L^2(1,\chi)\sum_e\mathop{\sum\sum}_{(u,v)=1} \frac{\rho(eu)\rho(ev)}{euv}g(eu)g(ev)J(u,v)\\
    &&+O(T^{-\frac{1}{9}}(\log{T})^6)\nonumber
  \end{eqnarray}
  where
  \begin{equation}
    \label{eq:17.9}
    J(u,v)=\int_0^\infty k\left(\frac{T}{x}\right)h\left(\frac{x}{u}\right)h\left(\frac{x}{v}\right)\frac{dx}{x}.
  \end{equation}
\end{lemma}

\begin{remarks*}
  We should have restricted $u,v$ in \eqref{eq:17.8} by the conditions \eqref{eq:16.9} which we ignored, because if $D\mid u$ or $D\mid v$, then some trivial estimations yield a small contribution. We shall address this issue in Section \ref{section:commentsCompletingProof}.
\end{remarks*}

It is convenient to treat the two parts $\gamma^*(d),\gamma'(d)$ separately, so we write \eqref{eq:17.8} in the following fashion
\begin{equation}
  \label{eq:17.10}
  \Kc^\neq(T)=L^2(1,\chi)(\Kc^*(T)-\Kc'(T))+O(T^{-\frac{1}{9}}(\log{T})^6).
\end{equation}
Here $\Kc^*$ stands for this multiple sum
\begin{equation}
  \label{eq:17.11}
  \Kc^*(T)=\sum_e\mathop{\sum\sum}_{(u,v)=1} \frac{\rho(eu)\rho(ev)}{euv}g(eu)g(ev)J^*(u,v)
\end{equation}
where $J^*(u,v)$ is defined by the integral \eqref{eq:17.9} with the kernel $k(y)$ replaced by
\begin{equation}
  \label{eq:17.12}
  k^*(y)=\sum_d  \gamma^*(d)\phi(dy).
\end{equation}
The second part $\Kc'(T)$ is defined in the same fashion but with $\gamma'(d)$ in place of $\gamma^*(d)$.

It is not surprising that the off-diagonal contribution $\Kc^\neq(T)$ gains the factor $L^2(1,\chi)$, because the lacunarity of $L(s)=\zeta(s)L(s,\chi)$ strikes independently two times. However, it is not enough gain to treat $\Kc^*$ and $\Kc'(T)$ crudely. We need to exploit important features of the mollifier which creates some sifting effects. The job would be quick if one used the Riemann hypothesis for $L(s)$, but of course, this is prohibited. We shall estimate $\Kc^*(T)$ unconditionally by delicate elementary arguments.

\begin{remarks*}
  The coefficients $\rho(m)g(m)$ of the mollifier \eqref{eq:3.13} are supported on cubefree numbers. Yet, it is technically messy to keep track of the square factors in the off-diagonal part, they play no essential role because we do not mind loosing absolute constants. One can pull out the square factors in the same fashion as we have shown for the diagonal terms in Section \ref{section:reducingSquarefreeDiagonalTerms}. This operation requires small changes in the range of variables of test functions, nevertheless the notation suffers. We leave for the prudent reader to fill up details while we are working on (from now on) under the assumption that the sum \eqref{eq:17.11} is restricted to $eu$ and $ev$ being squarefree. This convenient assumption makes
  \begin{equation}
    \label{eq:17.13}
    \rho(eu)=\mu(eu)\lambda(eu), \ \rho(ev)=\mu(ev)\lambda(ev).
  \end{equation}
\end{remarks*}

\section{Computing the Series $\phi(z)$}
\label{section:ComputingSeriesphi}

We begin by providing crude estimates for $\phi(z)$. It is easy to see directly from \eqref{eq:17.7} and \eqref{eq:13.5} that $z\phi(z)\ll 1$. Moreover, by Poisson's formula
\begin{equation}
  \label{eq:18.1}
  \phi(z)=-\Psi(0)+\frac{1}{z}\sum_m \Phi\left(\frac{m}{z}\right)\ll 1
\end{equation}
because $\Phi$ (the Fourier transform of $\Psi$) is compactly supported with $\Phi(0)=0$. Together we conclude
\begin{equation}
  \label{eq:18.2}
  \phi(z)\ll (1+z)^{-1}, \ \text{ if }z>0.
\end{equation}
Hence the series \eqref{eq:17.11} converges absolutely.

\begin{remarks*}
  Applying the Euler-McLaurin formula to \eqref{eq:18.1} one obtains the exact expression
  \begin{equation}
    \label{eq:18.3}
    z\phi(z)=\int_0^\infty \{\xi z\}\Phi'(\xi)d\xi
  \end{equation}
  which yields \eqref{eq:18.2} at once by the inequality $0\le \{x\}\le\min(1,x)$. The bound \eqref{eq:18.2} cannot be improved if $z$ is small, see the second term in \eqref{eq:18.8}.
\end{remarks*}

Another way of computing $\phi(z)$ goes by contour integration. By \eqref{eq:17.7} we get
\begin{equation}
  \label{eq:18.5}
  \phi(z)=\frac{1}{2\pi i}\int_{(2)} \tilde\Psi(s)\zeta(s)z^{-s}ds
\end{equation}
where $\tilde\Psi(s)$ is the Mellin transform of $\Psi(y)$, which in turn is the Fourier transform of $\Phi(x)$, see \eqref{eq:13.4}. By Mellin's inversion followed by partial integration we get the formula
\begin{equation}
  \label{eq:18.6}
  s(s+1)\tilde\Psi(s)=\int_0^\infty \Psi''(z)z^{s+1}dz.
\end{equation}
This yields analytic continuation of $s(s+1)\tilde\Psi(s)$ to the half-plane $\Re{s}>-2$. For $s=-1$ we find that
\[\int_0^\infty \Psi''(z)dz=-\Psi'(0)=0\]
by \eqref{eq:13.4}, so $\tilde\Psi(s)$ has no pole at $s=-1$. For $s=0$ we find that
\[\int_0^\infty \Psi''(z)zdz=-\int_0^\infty \Psi'(z)dz=\Psi(0)=\int \Phi(x)dx,\]
so $\tilde\Psi(s)$ has a simple pole at $s=0$ with residue $\Psi(0)$. Note that
\[\tilde\Psi(1)=\int_0^\infty \Psi(z)dz=2\pi\Phi(0)=0,\]
so the simple pole of $\zeta(s)$ is cancelled by the zero of $\tilde\Psi(s)$ at $s=1$ in \eqref{eq:18.5}. Hence we get
\begin{equation}
  \label{eq:18.7}
  \phi(z)=\frac{1}{2\pi i}\int_{(\varepsilon)}\tilde\Psi(s)\zeta(s)z^{-s}ds
\end{equation}
with any $\varepsilon>0$. It will be a more friendly expression for $\phi(z)$ if the pole at $s=0$ is removed. To this end we write \eqref{eq:18.7} in the following form
\begin{equation}
  \label{eq:18.8}
  \phi(z)=\phi_0(z)-\frac{1}{2}\Psi(0)(1-z)^+
\end{equation}
where
\begin{equation}
  \label{eq:18.9}
  \phi_0(z)=\frac{1}{2\pi i}\int_{(\varepsilon)}\theta(s)z^{-s}ds
\end{equation}
and
\begin{equation}
  \label{eq:18.10}
  \theta(s)=\tilde\Psi(s)\zeta(s)+\Psi(0)/2s(s+1).
\end{equation}
Note that $\theta(s)$ is holomorphic in $\Re{s}>-1$ and it satisfies
\begin{equation}
  \label{eq:18.11}
  (s+1)\theta(s)\ll (|s|+1)^{-1}, \ \text{ if }-1<\Re{s}\le A
\end{equation}
because $\zeta(0)=-\frac{1}{2}$ and $s(s+1)\tilde\Psi(s)\ll (|s|+1)^{-2}$ in vertical strips. By \eqref{eq:18.9} we derive
\begin{equation}
  \label{eq:18.12}
  \phi_0(z)\ll z(1+z)^{-A}, \ \text{ if }z>0,
\end{equation}
for any $A\ge 2$, the implied constant depending on $A$.

\section{Computing the Series $k^*(y)$}
\label{section:computingSeriesk*}

First we show a formula in a bit more general case. Recall that $\gamma^*(d)$ depends on $w$.

\begin{lemma}\label{lemma:19.1}
  If $w$ is squarefree, then
  \begin{equation}
    \label{eq:19.1}
    \sum_d \gamma^*(d)f(dy)=\frac{\xi(w)}{\zeta(2)}\sum_{c\mid w}\chi(c)\sum_{(d,D)=1} f(cdy)d^{-1}
  \end{equation}
  for any $f(y)$, provided the series $\sum f(dy)d^{-1}$ converges absolutely.
\end{lemma}
\begin{proof}
  The sequence $\gamma^*(d)$ has been computed in Lemma \ref{lemma:16.1}. By \eqref{eq:16.13} we proceed as follows
  \begin{eqnarray*}
    \zeta(s)\sum_d \gamma^*(d)f(dy)&=&\sum_{(d,D)=1} \frac{(d,w)}{d}\chi((d,w))\xi\left(\frac{w}{(d,w)}\right)f(dy)\\
    &=&\sum_{a\mid w} \chi(a)\xi\left(\frac{w}{a}\right)\sum_{(d,Dw/a)=1} f(ady)d^{-1}\\
    &=&\xi(w)\sum_{\substack{ab\mid w\\ (ab,D)=1}} \frac{\chi(a)\mu(b)}{\xi(a)b} \sum_{(d,D)=1} f(abdy)d^{-1}.
  \end{eqnarray*}
  Given $c\mid w$ with $(c,D)=1$, the above sum over $ab=c$ is equal to 
  \[\prod_{p\mid c} \left(\frac{\chi(p)}{\xi(p)}-\frac{1}{p}\right)=\prod_{p\mid c} \chi(p)=\chi(c).\]
  This completes the proof of \eqref{eq:19.1}.
\end{proof}
\begin{corollary}\label{cor:19.2}
  Suppose $w=uv$ is squarefree. For every $y>0$ we have
  \begin{equation}
    \label{eq:19.2}
    k^*(y)=\frac{\xi(w)}{\zeta(2)}\sum_{c\mid w}\chi(c)\sum_{(d,D)=1}\phi(cdy)d^{-1}.
  \end{equation}
\end{corollary}

It is difficult to execute the summation over $d$ in \eqref{eq:19.2} if $y\asymp 1$, so we shall use the formula \eqref{eq:19.2} as it is in its raw format. Nevertheless, regardless applications, we continue developing this formula further since the reader may like to see the shape of the function $k^*(y)$ from various perspectives such as \eqref{eq:19.11}. First, applying the partition \eqref{eq:18.8} to \eqref{eq:17.12} we derive
\begin{equation}
  \label{eq:19.3}
  k^*(y)=\frac{1}{2\pi i}\int_{(\varepsilon)}\theta(s)\zf^*(s)y^{-s}-\frac{1}{2}\Psi(0)\sum_{dy<1}\gamma^*(d)(1-dy),
\end{equation}
where $\zf^*(s)$ is the zeta function of the sequence $\gamma^*(d)$;
\begin{equation}
  \label{eq:19.4}
  \zf^*(s)=\sum_d \gamma^*(d)d^{-s}.
\end{equation}
Next, by Lemma \ref{lemma:19.1}, if $w$ is squarefree, then
\begin{equation}
  \label{eq:19.5}
  \zf^*(s)=\zeta(2)^{-1}\zeta_D(s+1)\xi(w)\prod_{p\mid w}\left(1+\frac{\chi(p)}{p^2}\right).
\end{equation}
This expression shows that $\zf^*(s)$ is analytic in the whole $s$-plane and it has only a simple pole at $s=0$ with the residue $r=\lambda(w)\xi(w)\varphi(D)/\zeta(2)D$. Introducing \eqref{eq:19.5} into \eqref{eq:19.3} and moving the integration from $\Re{s}=\varepsilon$ to $\Re{s}=-1$ we get
\begin{equation}
  \label{eq:19.6}
  k^*(y)=r\theta(0)-\frac{\Psi(0)\xi(w)}{2\zeta(2)}\sum_{c\mid w}\chi(c)\sum_{\substack{cdy<1\\ (d,D)=1}} (1-cdy)d^{-1}+O(yw\tau(D)).
\end{equation}
Furthermore, for any $X>0$ we have
\begin{equation}
  \label{eq:19.7}
  \sum_{\substack{d<X\\ (d,D)=1}} (1-d/X)d^{-1}=\frac{\varphi(D)}{D}\left(\log{X}+\gamma-1-\alpha(D)\right)+O(\tau(D)/X),
\end{equation}
where $\gamma=0.57\dots$ is the Euler constant and 
\begin{equation}
  \label{eq:19.8}
  \alpha(D)=\sum_{p\mid D} \frac{\log{p}}{p-1}\le \log{|D|}.
\end{equation}
Hence, the double sum in \eqref{eq:19.6} is equal to
\begin{equation}
  \label{eq:19.9}
  \frac{\varphi(D)}{D}\sum_{c\mid w}\chi(c)(-\log{cy}+\gamma-1-\alpha(D))+O(y\tau(D)\Sf_1(w)).
\end{equation}
If $\chi(w)=1$, which is our case by the mollifier support, then
\begin{equation}
  \label{eq:19.10}
  \sum_{c\mid w}\chi(c)\log{\sqrt{w}/c}=0.
\end{equation}
To see this we switch $c$ to its complementary divisor $w/c$. Hence the sum over $c/w$ in \eqref{eq:19.9} becomes $\lambda(w)(-\log{y\sqrt{w}}+\gamma-1-\alpha(D))$. Combining the above results we conclude the following approximate formula which is useful only if $yw$ is small.
\begin{lemma}\label{lemma:19.3}
  Suppose $w=uv$ is squarefree and $\chi(w)=1$. For every $y>0$ we have
  \begin{equation}
    \label{eq:19.11}
    k^*(y)=\lambda(w)\xi(w)\frac{\varphi(D)}{2\zeta(2)D}\{\Psi(0)\log{y\sqrt{w}}+\alpha(D)+\alpha_0\}+O(yw\tau(D))
  \end{equation}
  where $\alpha_0$ and the implied constant depend only on the fixed test function $\Phi$.
\end{lemma}

\section{Evaluation of $J^*(u,v)$}
\label{section:evaluationJ*}
Recall that $J^*(u,v)$ is the integral
\begin{equation}
  \label{eq:20.1}
  J^*(u,v)=\int_0^\infty k^*\left(\frac{T}{x}\right)h\left(\frac{x}{u}\right)h\left(\frac{x}{v}\right)\frac{dx}{x}
\end{equation}
with the kernel $k^*(y)$ given by \eqref{eq:17.12}. Applying \eqref{eq:19.2} we derive
\begin{equation}
  \label{eq:20.2}
  J^*(u,v)=\frac{\xi(w)}{\zeta(2)}\sum_{c\mid w}\chi(c)\sum_{(d,D)=1} d^{-1} \int_0^\infty \phi\left(\frac{dT}{x}\right)h\left(\frac{cx}{u}\right)h\left(\frac{cx}{v}\right)\frac{dx}{x}.
\end{equation}
In the sequel we shall use the following abbreviations:
\[\gamma=\frac{\log{c}}{\log{N}}, \ \gamma_1=\frac{\log{u}}{\log{N}}, \ \gamma_2=\frac{\log{r}}{\log{N}},\ \omega_1=\frac{\log{c/u}}{\log{N}}, \ \omega_2=\frac{\log{c/v}}{\log{N}}.\]
Therefore, $\gamma=\gamma_1+\omega_1=\gamma_2+\omega_2$. Note that all these five numbers are bounded in absolute value by
\[\mu=\frac{\log{M}}{\log{N}}<\beta-\frac{1}{2}\]
which is a small number. Moreover we shall be frequently changing the variable $x$ to $t=\log{x}/\log{N}$, so
\[x=N^t, \ x^{-1}dx=(\log{N})dt.\]
Now recall that in \eqref{eq:5.2} we set $h(x)=a(t)$ and in \eqref{eq:4.1} we requested $a(t)$ to be a smooth function on $\R$ with 
\begin{align*}
  a(t)=1-t,&\hspace{0.5cm}\text{ if }t\le\beta,\\
  0<a(t)<1-\beta,&\hspace{0.5cm}\text{ if }\beta<t<\alpha,\\
  a(t)=0,&\hspace{0.5cm}\text{ if }t\ge\alpha,
\end{align*}
where the transition points $\alpha>\beta>1/2$ are close to $1/2$. Moreover, recall that $N=Q^2T^2=|D|(T/2\pi)^2$ and $\log{|D|}/\log{T}$ is very small so $\log{N}$ is close to $2\log{T}$.

We break the integration at $x=TM^2=X$, say, and write respectively
\begin{equation}
  \label{eq:20.3}
  J^*(u,v)=\frac{\xi(w)}{\zeta(2)}\left(J_1(u,v)+J_2(u,v)\right).
\end{equation}
In the first part we have $h(cx/u)h(cx/v)=(1-t-\omega_1)(1-t-\omega_2)=(1-t)^2-(\omega_1+\omega_2)(1-t)+\omega_1\omega_2$. Note that
\[\sum_{c\mid w} \chi(c)=\lambda(w), \ \sum_{c\mid w}\chi(c)(\omega_1+\omega_2)=0, \ \sum_{c\mid w}\chi(c)\omega_1\omega_2=\lambda(u,v)(\log{N})^{-2}\]
where
\begin{equation}
  \label{eq:20.4}
  \lambda(u,v)=\sum_{c\mid uv}\chi(c)\log{\frac{c}{u}}\log{\frac{c}{v}}.
\end{equation}
The vanishing of the middle sum above follows from \eqref{eq:19.10}. We shall compute $\lambda(u,v)$ and other alike arithmetic functions in the next section. Now we have
\begin{eqnarray}
 J_1(u,v)&=&\sum_{(d,D)=1}d^{-1}\int_0^X\phi\left(\frac{dT}{x}\right)\left(\lambda(w)(1-t)^2+\lambda(u,v)(\log{N})^{-2}\right)\frac{dx}{x}\nonumber\\
 \label{eq:20.5} &=&\lambda(w)A_0(\log{N})^2+\lambda(u,v)A_1,
\end{eqnarray}
with $A_0\ll \varphi(D)/|D|$ and $A_1\ll \varphi(D)/|D|$ which are independent of $u,v$.

In the second part $J_2(u,v)$ the integration starts at $x=X$ so there is a room for the variable $d$. We apply \eqref{eq:18.8} and \eqref{eq:19.7} getting
\begin{eqnarray*}
  &&\sum_{(d,D)=1}d^{-1}\phi\left(\frac{dT}{x}\right)=\sum_{(d,D)=1}d^{-1}\left(\phi_0\left(\frac{dT}{x}\right)-\frac{1}{2}\Psi(0)\left(1-\frac{dT}{x}\right)^t\right)\\
  &=&\frac{\varphi(D)}{D}\int_0^\infty \phi_0(z)\frac{dz}{z}-\Psi(0)\frac{\varphi(D)}{2D}\left(\log{\frac{x}{T}}+\gamma-1+\alpha(D)\right)\\
  &&+O(\tau(D)T/x)\\
  &=&A\log{N}-B\log{x}+O(\tau(D)T/x),
\end{eqnarray*}
say, with
\begin{equation}
  \label{eq:20.6}
  A\log{N}=\frac{\varphi(D)}{D}\int_0^\infty \phi_0(z)z^{-1}dz+\Psi(0)\frac{\varphi(D)}{2D}(\log{T}-\gamma+1-\alpha(D))
\end{equation}
and
\begin{equation}
  \label{eq:20.7}
  B=\Psi(0)\varphi(D)/2D.
\end{equation}
We have bounds
\begin{equation}
  \label{eq:20.8}
  A\ll \varphi(D)/|D|, \hspace{0.5cm} B\ll\varphi(D)/|D|
\end{equation}
and we need nothing else to know about $A,B$. The error term $O(\tau(D)T/x)$ contributes to $J^\infty(u,v)$ at most 
\[\tau(w)\tau(D)T\int_X^\infty x^{-2}dx=\tau(w)\tau(D)M^{-2}<T^{-1/400}.\]
By the above estimates we get $J_2(u,v)=J_{20}(u,v)+O(T^{-1/400})$ with $J_{20}(u,v)$ equal to
\begin{eqnarray*}
  &&\sum_{c\mid w} \chi(c)\int_X^\infty \left(A\log{N}-B\log{x}\right)h\left(\frac{cx}{u}\right)h\left(\frac{cx}{v}\right)\frac{dx}{x}\\
  &=&\sum_{c\mid w}\chi(c)\int_{Xc/\sqrt{w}}^\infty \left(A\log{N}-B\log{\frac{x\sqrt{w}}{c}}\right)h\left(x\sqrt{u/v}\right)h\left(x\sqrt{v/u}\right)\frac{dx}{x}.
\end{eqnarray*}
If the integration starts from $X$ we get an elegant quantity (see \eqref{eq:19.10})
\begin{equation}
  \label{eq:20.9}
  K(u/v)=\int_X^\infty \left(A-B\log{x}/\log{N}\right)h(x\sqrt{u/v})h(x\sqrt{v/u})x^{-1}dx
\end{equation}
and
\begin{equation}
  \label{eq:20.10}
  J_{20}(u,v)=\lambda(w)K(u/v)\log{N}.
\end{equation}
Estimating trivially we get $K(u/v)\ll\log{N}$. This bound has true order of magnitude, but it is not useful, because we shall need a clear view on the dependence on $u/v$. The remaining part is equal to
\begin{eqnarray*}
  J_{22}(u,v)=\sum_{c\mid w}\chi(c)\int_{Xc/\sqrt{w}}^X &&\left(A\log{N}-B\log{\frac{x\sqrt{w}}{c}}\right)\\
  &&\left[\left(1-\frac{\log{x}}{\log{N}}\right)^2-\left(\frac{\log{u/v}}{2\log{N}}\right)^2\right]\frac{dx}{x}.
\end{eqnarray*}
Put
\[\nu=\log{X}/\log{N}=2\mu+\log{T}/\log{N}=2\mu+\frac{1}{2}\left(1+\frac{\log{Q}}{\log{T}}\right)^{-1},\]
\[\delta=\frac{1}{2}(\omega_1+\omega_2)=\frac{1}{2}\log(c/\sqrt{w})/\log{N}.\]
In this notation we have
\begin{equation}
  \label{eq:20.11}
  J_{22}(u,v)=(\log{N})^2\sum_{c\mid w}\chi(c)P(\delta)
\end{equation}
where $P(\delta)$ is the polynomial in $\delta$ of degree five given by
\begin{eqnarray}
  P(\delta)&=&\int_{\nu+\delta}^\nu (A-Bt+B\delta)\left((1-t)^2-\frac{1}{4}(\gamma_1-\gamma_2)^2\right)\nonumber\\ \label{eq:20.12}&=&\frac{3B}{8}(\gamma_1-\gamma_2)^2\delta^2+\frac{1-\nu}{2}(2A-3B+B\nu)\delta^2-\frac{B}{12}\delta^4\\
  &&+\text{ odd degree monomials}.\nonumber
\end{eqnarray}
Since
\[\sum_{c\mid w} \chi(c)\delta^j=0, \ \text{ if }2\nmid j,\]
we do not need the odd degree monomials. We get
\begin{equation}
  \label{eq:20.13}
  J_{22}(u,v)=\lambda_2(uv)A_2+\lambda_4(uv)A_4(\log{N})^{-2}
\end{equation}
where
\begin{equation}
  \label{eq:20.14}
  \lambda_j(w)=\sum_{c\mid w}\chi(c)(\log{c/\sqrt{w}})^j
\end{equation}
and $A_2\ll \varphi(D)/|D|, \ A_4\ll \varphi(D)/|D|$ do not depend on $u,v$. Gathering the above results we arrive at the following representation of $J^*(u,v)$.

\begin{lemma}\label{lemma:20.1}
  Suppose $w=uv$ is squarefree and $\chi(w)=1$. Then
  \begin{eqnarray} J^*(u,v)\zeta(2)/\xi(w)&=&\lambda(w)K(u/v)\log{N}+\lambda(w)A_0(\log{N})^2\nonumber\\
    \label{eq:20.15}   &&+\lambda(u,v)A_1+\lambda_2(w)A_2+\lambda_4(w)A_4(\log{N})^{-2}\\
    &&+O(T^{-1/400}),\nonumber
  \end{eqnarray}
  where $K(u/v)$, $\lambda(u,v)$, $\lambda_2(w)$, $\lambda_4(w)$ are given by \eqref{eq:20.9}, \eqref{eq:20.4}, \eqref{eq:20.14}, respectively. Moreover $A_0,A_1,A_2,A_4$ do not depend on $u,v$ and they are $\ll\varphi(D)/|D|$.
\end{lemma}

\begin{remarks*}
  It is essential that $K(u/v)$ depends on the ratio $u/v$ rather than on $u,v$ respectively. After computing $\lambda(u,v),\lambda_2(w),\lambda_4(w)$ in the next section we shall see that all the terms in \eqref{eq:20.15} look alike and each one has the order of magnitude $\lambda(w)(\log{N})^2$ (except for the negligible error term). The formula \eqref{eq:20.15} displays the behaviour in terms of $u,v$ as needed, but it is long, so we wish to say that our arrangements could have been quicker if we applied Taylor's expansion of $a(t)$. This would bring polynomials in $\delta$ of arbitrary degree; consequently we would have struggled with the uniformity in the resulting series coefficients, which is a formidable task. The fact that we are dealing here with $P(\delta)$ of degree five is due to the linearity of $a(t)$ in the long segment $t\le\beta$ with $\beta$ slightly larger than $1/2$.
\end{remarks*}

\section{Computing the $\lambda$-functions}
\label{section:computingLambdaFunctions}
Recall that $\lambda(u,v)$ and $\lambda_j(w)$ are defined by convolutions of $\chi$ against powers of logarithms. In this sections we arrange these as convolutions of $1$ against the von Mangoldt functions
\[\Lambda_j=\mu\star (\log)^j.\]
First writing $(\log{c/u})(\log{c/v})=(\log{c/\sqrt{w}})^2-(\log{\sqrt{u/v}})^2$ in \eqref{eq:20.4} we find that
\begin{equation}
  \label{eq:21.1}
  \lambda(u,v)=\lambda_2(w)-\lambda(w)\left(\frac{1}{2}\log{\frac{u}{v}}\right)^2.
\end{equation}
Next, writing
\[\lambda_j(uv)=\sum_{c\mid u}\sum_{d\mid v}\chi(cd)\left(\log{\frac{c}{\sqrt{u}}}+\log{\frac{d}{\sqrt{v}}}\right)^j\]
we find that
\begin{equation}
  \label{eq:21.2}
  \lambda_j(uv)=\sum_{a+b=j}\binom{j}{a}\lambda_a(u)\lambda_b(v).
\end{equation}
Observe that $\lambda_a(u)=0$ if $a$ is odd, so $a,b$ run in \eqref{eq:21.2} over even numbers. For example we get
\begin{equation}
  \label{eq:21.3}
  \lambda_2(uv)=\lambda(v)\lambda_2(u)+\lambda(u)\lambda_2(v)
\end{equation}
\begin{equation}
  \label{eq:21.4}
  \lambda_4(uv)=\lambda(v)\lambda_4(u)+6\lambda_2(u)\lambda_2(v)+\lambda(u)\lambda_4(v).
\end{equation}

We shall attach the $\lambda$-functions to the mollifier factors $\rho(eu)\rho(ev)$ which vanish if $\lambda(p)=0$ for some $p\mid euv$, see \eqref{eq:17.13}. Therefore, for computing $\lambda_j(u)$ we can assume that $\lambda(u)\not=0$, in which case
\begin{equation}
  \label{eq:21.5}
  \lambda(q)=\tau(q/(q,D)),\text{ if }q\mid u.
\end{equation}

We compute $\lambda_j(u)$ as follows
\begin{eqnarray*}
  2^j\lambda_j(u)&=&\sum_{c\mid u}\chi(c)\left(\log{c}-\log{\frac{u}{c}}\right)^j\\
  &=&\sum_{a+b=j}\binom{j}{a}(-1)^b(\chi\log^a)\star(\log^b).
\end{eqnarray*}
Here we write $(\chi\log^a)\star(\log^b)=\chi(1\star\Lambda_a)\star(1\star\Lambda_b)=\chi\star 1\star \chi\Lambda_a\star\Lambda_b=\lambda\star\chi\Lambda_a\star\Lambda_b$, and then
\[\sum_{\ell mn=u} \lambda(\ell)\chi(m)\Lambda_a(m)\Lambda_b(n)=\lambda(u)\sum_{mn\mid u} \chi(m) \frac{\tau(n,D)}{\tau(mn)}\Lambda_a(m)\Lambda_b(n).\]
Adding the above expressions we obtain
\begin{lemma}\label{lemma:21.1}
  Suppose $u$ is squarefree with $\lambda(u)\not=0$. Then we have
  \begin{equation}
    \label{eq:21.6}
    \lambda_j(u)=\lambda(u)\sum_{q\mid u}\Lambda_j^*(q)
  \end{equation}
  where
  \begin{equation}
    \label{eq:21.7}
    \Lambda_j^*(q)=\frac{\tau((q,D))}{2^j\tau(q)}\sum_{mn=q}\chi(m)\sum_{a+b=j}\binom{j}{a}(-1)^b\Lambda_a(m)\Lambda_b(n).
  \end{equation}
\end{lemma}
We do not need to know $\Lambda_j^*(q)$ exactly, the following estimate is good enough
\begin{equation}
  \label{eq:21.8}
  |\Lambda^*_j(q)|\le 2^{-j}\Lambda_j(q).
\end{equation}
Hence $\Lambda_j^*(q)$ is supported on numbers having at most $j$ distinct prime divisors. Moreover we get
\begin{equation}
  \label{eq:21.9}
  \sum_{q\le x}|\Lambda^*(q)|q^{-1}\ll (\log{x})^j.
\end{equation}
Note that \eqref{eq:21.6} holds for $\lambda_0(u)=\lambda(u)$ with $\Lambda_0^*(q)=\Lambda_0(q)$.

Inserting \eqref{eq:21.6} into \eqref{eq:21.2} we obtain the following result
\begin{lemma}\label{lemma:21.2}
  Suppose $w=uv$ is squarefree with $\lambda(w)\not=0$. Then we have
  \begin{equation}
    \label{eq:21.10}
    \lambda_j(w)=\lambda(w)\sum_{a+b=j}\binom{j}{a}\sum_{q\mid u}\sum_{r\mid v} \Lambda_a^*(q)\Lambda_b^*(r).
  \end{equation}
\end{lemma}

The second part of \eqref{eq:21.1} can be written in the same fashion, exactly we have
\begin{equation}
  \label{eq:21.111}
  \left(\log{\frac{u}{v}}\right)^2=\sum_{q\mid u}\Lambda_2(q)-2\sum_{q\mid u}\sum_{r\mid v}\Lambda(q)\Lambda(r)+\sum_{r\mid v}\Lambda_2(r).
\end{equation}
By \eqref{eq:21.3} and \eqref{eq:21.6} we get
\begin{equation}
  \label{eq:21.12}
  \lambda_2(w)=\lambda(w)\left(\sum_{q\mid u}\Lambda_2^*(q)+\sum_{r\mid v}\Lambda_2^*(r)\right)
\end{equation}
and
\begin{lemma}\label{lemma:21.3}
  Suppose $w=uv$ is squarefree with $\lambda(w)\not=0$. Then we have
  \begin{eqnarray}
    \lambda(u,v)=\lambda(w)\left\{\sum_{q\mid u}\left(\Lambda_2^*(q)-\frac{1}{4}\Lambda_2(q)\right)\right.&+&\sum_{r\mid v}\left(\Lambda_2^*(r)-\frac{1}{4}\Lambda_2(r)\right)\nonumber\\
    \label{eq:2}&+&\left.\frac{1}{2}\sum_{q\mid u}\sum_{r\mid v}\Lambda(q)\Lambda(r)\right\}.
  \end{eqnarray}
\end{lemma}

We conclude this section by combining the results into a compact implicit form
\begin{lemma}\label{lemma:21.4}
  Suppose $w=uv$ is squarefree with $\lambda(w)\not=0$. Then
  \begin{eqnarray}
    J^*(u,v)=&&\lambda(w)\xi(w)(\log{N})^2/\zeta(2)\nonumber\\
    \label{eq:21.13}&&\left\{ \frac{K(u,v)}{\log{N}}+\sum_{a+b\le 4}c(a,b)(\log{N})^{-a-b}\sum_{q\mid u}\sum_{r\mid v}\Lambda_a^*(q)\Lambda_b^*(r)\right\}
  \end{eqnarray}
  where $K(u,v)$ is given by \eqref{eq:20.9} and the coefficients $c(a,b)$ do not depend on $u,v$ and they satisfy $c(a,b)\ll \varphi(D)/|D|$. Moreover $\Lambda_a^*(q), \Lambda_b^*(r)$ given by \eqref{eq:21.7} are supported on numbers having at most $a,b$ prime factors, respectively and they satisfy the bound \eqref{eq:21.8}, a fortiori \eqref{eq:21.9}.
\end{lemma}

\section{Estimating $E_{ab}$}
\label{section:EstimatingEab}

According to \eqref{eq:21.13} the formula \eqref{eq:17.11} splits into
\begin{eqnarray}
  \Kc^*(T)=&&\left\{F(\log{N})^{-1}+\sum_{a+b\le 4} c(a,b)(\log{N})^{-a-b}E_{ab}\right\}(\log{N})^2\zeta(2)^{-1}\nonumber\\
  \label{eq:22.1}&&+O(T^{-1/401})
\end{eqnarray}
where
\begin{equation}
  \label{eq:22.2}
  E=\sum_e\mathop{\sum\sum}_{(u,v)=1}\frac{\rho(eu)\rho(ev)}{euv} g(eu)g(ev)\lambda(uv)\xi(uv)K(u/v),
\end{equation}
and
\begin{equation}
  \label{eq:22.3}
  E_{ab}=\sum_e\mathop{\sum\sum}_{(u,v)=1}\frac{\rho(eu)\rho(ev)}{euv}g(eu)g(ev)\lambda(uv)\xi(uv)\sum_{q\mid u}\Lambda_a^*(q)\sum_{r\mid v}\Lambda_b^*(r).
\end{equation}
Many arguments for estimating $E$ and $E_{ab}$ are reminiscent of these applied to the diagonal terms in early sections. All cases are similar, but $E$ needs extra attention.

First we do $E_{00}$, because it is a notationally simpler model for every $E_{ab}$. In this case \eqref{eq:22.3} reduces to
\begin{equation}
  \label{eq:22.4}
  E_{00}=\sum_e\mathop{\sum\sum}_{(u,v)=1}\frac{\rho(eu)\rho(ev)}{euv} g(eu)g(ev)\lambda(uv)\xi(uv).
\end{equation}
Observe that the total contribution of $E_{00}$ to the off-diagonal part $\Kc^\neq(T)$ (see \eqref{eq:17.10} and \eqref{eq:17.11}) is equal to
\begin{equation}
  \label{eq:22.5}
  V_{00}=c(0,0)\zeta(2)^{-1}(L(1,\chi)\log{N})^2E_{00}.
\end{equation}
If $E_{00}$ is shown to be bounded (as expected due to the sifting effects) the small factor $L(1,\chi)\log{N}$ (due to the lacunarity effect) yields the vital saving two times. However, it is hard to show that $E_{00}\ll 1$, because twisting the Möbius function with the character $\chi$ does not exactly work that way in our exceptional scenario. An attempt to execute the summation in $E_{00}$ by routine arguments fails as one cannot keep track in the conductor aspect. Therefore we take a roundabout path. We shall loose some portion of the saving factor $(L(1,\chi)\log{N})^2$, but fortunately not the whole saving.

We start by reconstructing $E_{00}$ from the following expression
\begin{equation}
  \label{eq:22.6}
  W=\left(\sum_{N<\ell\le N^2}-\sum_{N^2<\ell\le N^3}\right)\frac{\tilde\lambda(\ell)}{\ell}\left(\sum_{m\mid \ell}\mu(m)\frac{g(m)}{\xi(m)}\right)^2
\end{equation}
where $\tilde\lambda(\ell)$ is the completely multiplicative function with
\begin{equation}
  \label{eq:22.7}
  \tilde\lambda(p)=(\lambda(p)\xi(p))^2=(1+\chi(p))^2(1+\chi(p)/p)^{-2}.
\end{equation}
Opening the square we get
\[W=\sum_{e}\mathop{\sum\sum}_{(u,v)=1}\frac{\mu(eu)\mu(ev)}{euv}\frac{g(eu)g(ev)}{\xi(ev)\xi(ev)}\tilde\lambda(euv)\left(\sum_k-\sum_k\right)\frac{\tilde\lambda(k)}{k}\]
where $k$ runs over the segments $N/euv<k\le N^2/euv$ and $N^2/euv<k\le N^3/euv$, respectively. Note that the above triple sum preceding the sums over $k$ matches that in \eqref{eq:22.4} (check this using \eqref{eq:17.13} and our choice \eqref{eq:22.7}).

The generating Dirichlet series of $\tilde\lambda(k)$ is
\[\tilde L(s)=\sum \tilde\lambda(k)k^{-s}=\prod_p (1-\tilde\lambda(p)p^{-s})^{-1}=L(s)^2\tilde R(s)\]
where the ``correcting'' factor $\tilde R(s)$ is given by the Euler product
\[\tilde R(s)=\prod_p(1-\tilde\lambda(p)p^{-s})^{-1}(1-p^{-s})^2(1-\chi(p)p^{-s})^2\]
which converges absolutely in $\Re{s}>\frac{1}{2}$. For $s=1$ we have
\[\tilde R(1)=\frac{1}{\zeta(2)}\prod_{p\mid D}\left(1+\frac{1}{p}\right)\asymp \frac{\varphi(D)}{|D|}.\]
Hence, by standard contour integration, we derive
\begin{eqnarray*}
  \left(\sum_k-\sum_k\right)\frac{\tilde\lambda(k)}{k}&=&\res_{s=0}\tilde L(s+1)s^{-1}(N^{2s}-N^s-N^{3s}+N^{2s})(euv)^{-s}\\
  &&+O(N^{-1/4})=-\tilde R(1)(L(1,\chi)\log{N})^2+O(N^{-1/4}).
\end{eqnarray*}
Hence
\begin{equation}
  \label{eq:22.8}
  W=-\tilde R(1)(L(1,\chi)\log{N})^2E_{00}+O(T^{-1/4}).
\end{equation}
It looks like we have lost the entire saving factor $(L(1,\chi)\log{N})^2$. Not true, because $W$ is small due to the lacunarity of $\tilde\lambda(\ell)$. To estimate $W$ we proceed along the lines in Section \ref{section:estimatingSflat} getting
\[W\ll \sum_{N<\ell\le N^3}\tilde\lambda(\ell)\ell^{-1}\big(\sum_{m\mid\ell} \ \big)^2\ll \big(\sum_\ell \lambda(\ell)\ell^{-1}\big)^{\frac{1}{2}}\big(\sum_\ell \tau(\ell)^4\ell^{-1}\big(\sum_{m\mid\ell} \ \big)^2\big)^{\frac{1}{2}}.\]
Hence, by the same arguments as in the proof of Corollary \eqref{eq:9.2} we get
\begin{equation}
  \label{eq:22.9}
  W\ll \left(L(1,\chi)\log{N}\right)^{\frac{1}{2}}.
\end{equation}
The almost constant multiplicative function $1/\xi(m)$ in \eqref{eq:22.6} makes no difference to the arguments ($1/\xi(p)=1+\chi(p)/p$). Combining \eqref{eq:22.8} with \eqref{eq:22.9} we conclude that
\begin{equation}
  \label{eq:22.10}
  E_{00}\ll \frac{|D|}{\varphi(D)}\left(L(1,\chi)\log{N}\right)^{-\frac{3}{2}}.
\end{equation}
Finally, introducing \eqref{eq:22.10} into \eqref{eq:22.5} we get (recall $c(0,0)\ll \varphi(D)/|D|$)
\begin{equation}
  \label{eq:22.11}
  V_{00}\ll \left(L(1,\chi)\log{N}\right)^{\frac{1}{2}}.
\end{equation}

The other sums $E_{ab}$ can be reduced to $E_{00}$ by scaling the crop function of the mollifier. Specifically we write
\[E_{ab}=\mathop{\sum\sum}_{\substack{(q,r)=1\\ q,r<M}} \Lambda_a^*(q)\Lambda_b^*(r)\frac{\rho(q)\rho(r)}{qr}\lambda(qr)\xi(qr)E_{00}(q,r)\]
where $E_{00}(q,r)$ stands for $E_{00}$ with the crop functions $g(eu),g(ev)$ replaced by $g(equ),g(er v)$ and the summation variables are restricted by the co-primality condition: $(euv,qr)=1$. This co-primality condition does not really spoil the previous treatment of $E_{00}$ and the fact that the scaled down crop functions $g_q(m)=g(qm)$, $g_r(n)=g(r n)$ are not equal does not matter neither (the subsequent application of Cauchy's inequality resolves this discrepancy). However we need to address the scaling effect. Writing
\[g_q(m)=g(qm)=\left(\frac{\log{M/q}}{\log{M}}\right)^r \left(\frac{\log^+{M/qm}}{\log{M/q}}\right)^r\]
it boils down to changing the support range $M$ into $M/q$ and correcting the relevant estimates by factor $((\log{M/q})/\log{M})^r$. Well, not in full range, because Lemma \ref{lemma:9.1} used for estimating $W$ requires $Q^4<M$, which condition translates into $Q^4<M/q$. But in the range $M/q\le Q^4$ we can apply Lemma \ref{lemma:7.2} (precisely its relevant analogue). Adding the resulting estimates, we derive the following bound
\[W_{ab}(q)\ll \left(L(1,\chi)\log{N}\right)^{\frac{1}{2}}\left(\frac{\log{M/q}}{\log{M}}\frac{\log{N}}{\log{M/q}}\right)^{\frac{r}{2}}\ll \left(L(1,\chi)\log{N}\right)^{\frac{1}{2}}.\]
The same bound holds for $W_{ab}(r)$ (here $W_{ab}(q)$ and $W_{ab}(r)$ denote the sums of type $W$ with the crop function $g$ replaced by $g_q$ and $g_r$, respectively). Moreover we have
\[\left(\sum_{q<M}\Lambda_a^*(q)q^{-1}\right)\left(\sum_{r<M} \Lambda_b^*(r)r^{-1}\right)\ll (\log{M})^{a+b}.\]
Hence
\begin{equation}
  \label{eq:22.12}
  E_{ab}\ll \frac{|D|}{\varphi(D)}(\log{M})^{a+b}\left(L(1,\chi)\log{N}\right)^{\frac{3}{2}}
\end{equation}
and the total contribution of $E_{ab}$ to the off-diagonal part $\Kc^\neq(T)$, say $V_{ab}$, (see \eqref{eq:17.10}, \eqref{eq:17.11}, \eqref{eq:22.1}) satisfies
\begin{equation}
  \label{eq:22.13}
  V_{ab}\ll \left(L(1,\chi)\log{N}\right)^{\frac{1}{2}}.
\end{equation}

\section{Estimating $E$}
\label{section:estimatingE}

The case of $E$ can be regarded as a generalization of $E_{00}$ in which the extra kernel function $K(u/v)$ is introduced. We reconstruct $E$ from the following expression
\begin{equation}
  \label{eq:23.1}
  W=\left(\sum_{N<\ell\le N^2}-\sum_{N^2<\ell\le N^3}\right)\frac{\tilde\lambda(\ell)}{\ell}\sum_{m_1\mid\ell}\sum_{m_2\mid\ell}\mu(m_1)\mu(m_2)\frac{g(m_1)}{\xi(m_1)}\frac{g(m_2)}{\xi(m_2)}K\left(\frac{m_1}{m_2}\right).
\end{equation}
The same arguments which produced \eqref{eq:22.8} now yield
\begin{equation}
  \label{eq:23.2}
  W=-\tilde R(1)\left(L(1,\chi)\log{N}\right)^2E+O(T^{-1/4}).
\end{equation}

Before estimating $W$ we need to separate $m_1,m_2$ in $K(m_1/m_2)$. This can be accomplished quickly by changing the variable of integration $x$ in \eqref{eq:20.9} into $x/\sqrt{m_1m_2}$ giving
\[K\left(\frac{m_1}{m_2}\right)=\int_{X \sqrt{m_1m_2}}^N \left(A-B \frac{\log{x/\sqrt{m_1m_2}}}{\log{N}}\right)h\left(\frac{x}{m_1}\right)h\left(\frac{x}{m_2}\right)\frac{dx}{x}.\]

Recall that $N=Q^2T^2$ and $X=M^2T$. Since $h(x)=a(t)$ is linear in $t=\log{x}/\log{N}$ for $t<\beta$ we can write
\begin{eqnarray}
  K\left(\frac{m_1}{m_2}\right)=\int_X^N &&\left(A-B \frac{\log{x}}{\log{N}}+\frac{B}{2}\frac{\log{m_1}}{\log{N}}+\frac{B}{2}\frac{\log{m_2}}{\log{N}}\right)\nonumber\\
  \label{eq:23.3}&&h\left(\frac{x}{m_1}\right)h\left(\frac{x}{m_2}\right)\frac{dx}{x}+P \left(\frac{\log{m_1}}{\log{N}},\frac{\log{m_2}}{\log{N}}\right)\log{M}
\end{eqnarray}
where $P(x_1,x_2)$ is a polynomial of degree five with coefficients $\ll \varphi(D)/|D|$ (they are linear forms in $A,B$). From the above expressions the following convolution sums emerge (a kind of sifting weights);
\begin{equation}
  \label{eq:23.4}
  \theta_a(\ell)=\sum_{m\mid\ell} g(m)\frac{\mu(m)}{\xi(m)}\left(\frac{\log{m}}{\log{N}}\right)^a \hspace{0.4cm} (0\le a\le 5),
\end{equation}
\begin{equation}
  \label{eq:23.5}
  \theta_a(\ell;x)=\sum_{m\mid\ell} g(m)\frac{\mu(m)}{\xi(m)}\left(\frac{\log{m}}{\log{N}}\right)^a h\left(\frac{x}{m}\right) \hspace{0.4cm} (0\le a\le 1).
\end{equation}
Hence the double sum over the divisors of $\ell$ in \eqref{eq:23.1} splits into the integral
\begin{equation}
  \label{eq:23.6}
  \int_X^N \left[\left(A-B \frac{\log{x}}{\log{N}}\right)\theta_0(\ell;x)+B\theta_1(\ell;x)\right]\theta_0(\ell;x)\frac{dx}{x}
\end{equation}
and 36 terms of type $c(a,b)\theta_a(\ell)\theta_b(\ell)\log{M}$ for $0\le a,b\le 5$ with coefficients $c(a,b)\ll\varphi(D)/|D|$. Hence \eqref{eq:23.1} yields
\[W\ll\sum_{N<\ell<N^3} \tilde\lambda(\ell)\ell^{-1}(\theta_a(\ell)^2+\theta_b(\ell;x)^2)\frac{\varphi(D)}{|D|}\log{N} \]
for some $0\le a\le 5$, $0\le b\le 1$ and $X<x<N$. By the arguments in Section \ref{section:estimationT} we show that (compare it with Lemma \ref{lemma:9.1})
\begin{equation}
  \label{eq:23.8}
  \sum_{N<\ell <N^3} \tau(\ell)^4\ell^{-1}\left(\theta_a(\ell)^2+\theta_b(\ell;x)^2\right)\ll 1.
\end{equation}
Hence we derive in the same way as \eqref{eq:22.9} that
\begin{equation}
  \label{eq:23.9}
  W\ll \left(L(1,\chi)\log{N}\right)^{\frac{1}{2}}\frac{\varphi(D)}{|D|}\log{N}.
\end{equation}
On the other hand we have the formula \eqref{eq:23.2}, comparing these we get
\begin{equation}
  \label{eq:23.10}
  E\ll \left(L(1,\chi)\log{N}\right)^{-\frac{3}{2}}\log{N}.
\end{equation}
Finally it shows (see \eqref{eq:17.10} and \eqref{eq:22.1}) that the total contribution of $E$ to the off-diagonal $\Kc^\neq(T)$, say $V$, satisfies
\begin{equation}
  \label{eq:23.11}
  V\ll \left(L(1,\chi)\log{N}\right)^{\frac{1}{2}}.
\end{equation}

\section{Comments about Completing the Proof}
\label{section:commentsCompletingProof}

After having completed the treatment of the diagonal terms we wrapped up the results in the lower bound \eqref{eq:12.1} for $N_{00}(T)$ in which the off-diagonal contribution $\Kc^\neq(T)$ is postponed for handling in the rest of the paper. We are now ready to finish the job by compiling the relevant results.

According to the formula \eqref{eq:16.5} we partitioned $\Kc^\neq(T)$
into two similar parts $L^2(1,\chi)\Kc^*(T)-L^2(1,\chi)\Kc'(T)$, see
\eqref{eq:17.10} for exact formula. The third part was eliminated
earlier by making the assumption \eqref{eq:16.9}. This means that the
mollifier misses terms which are supported on multiples of
$D$. However the contribution of the missing terms can be treated by
undemanding arguments. For example one can show by straightforward
estimations that the missing part in the integral $I(T)$ (see
\eqref{eq:3.8}) is bounded by $|D|^{-\frac{1}{4}}T(\log{T})^4$ which
is negligible.

Next $\Kc^*(T)$ was split further into a number of pieces of similar shape (see \eqref{eq:22.1}) and the contribution to $\Kc^\neq(T)$ of every piece separately was shown to satisfy the same bound \eqref{eq:22.11}, \eqref{eq:22.13}, \eqref{eq:23.11}. This bound does not exceed the existing error term in the lower bound \eqref{eq:12.1}. Finally it remains to cover $\Kc'(T)$. We have gone quite far towards $\Kc'(T)$ by computing its constituents (see Lemma \ref{lemma:16.2}) until the case appeared merging the lines of $\Kc^*(T)$. Without repeating many of the same arguments we claim that $\Kc'(T)$ contributes no more than $\Kc^*(T)$.

\bibliographystyle{alpha}
\nocite{*}
\bibliography{references1}
\end{document}